\documentclass[11pt]{amsart}
 \usepackage{graphicx}
 \usepackage{amssymb}
 \usepackage{amsmath}
 \usepackage{amsthm}
 \usepackage{mathtools}
 \usepackage{tikz}
 \usepackage{psfrag}
 \usepackage{xcolor}

\newtheorem{thm}{Theorem}[section]
\newtheorem{lem}[thm]{Lemma}
\newtheorem{pro}[thm]{Proposition}
\newtheorem{cor}[thm]{Corollary}

\theoremstyle{definition}
\newtheorem{defn}{Definition}[section]
\newtheorem{remk}{Remark}[section]

\newcommand{\M}{\mathbb M}

\newcommand{\R}{\mathbb R}
\newcommand{\dist}{\operatorname{dist}}

\newcommand{\z}{\mathbf z}

\newcommand{\tr}{\operatorname{tr}}
\newcommand{\rank}{\operatorname{rank}}

\newcommand{\supp}{\operatorname{supp}}
\newcommand{\dv}{\operatorname{div}}

\makeatletter
\numberwithin{equation}{section}
\makeatother

\begin{document}

 \author{Baisheng Yan}
 \address{Department of Mathematics\\ Michigan State University\\ East Lansing, MI 48824, USA}
   \email{yanb@msu.edu}

\title[Nonuniqueness and instability for diffusion equations]{Convex integration for diffusion equations, I: Nonuniqueness and instability}


\begin{abstract}
We study the initial-boundary value problem for a  class of  diffusion equations with nonmonotone diffusion flux functions, including forward-backward parabolic equations and the gradient flows of  nonconvex energy functionals, under the framework of  partial differential inclusions using  the method of convex integration and Baire's category.   In  connection with  rank-one convex hulls of the corresponding matrix sets, we introduce a structural condition on the  
diffusion flux function, called  Condition (OC),  and establish  the nonuniqueness and instability of Lipschitz solutions to the initial-boundary value problem under this condition.  
 \end{abstract}

\maketitle

\section{Introduction}

Let $m,n\ge 1$, $T>0$, and let $\Omega$ be a bounded domain in $\R^n$ with Lipschitz boundary $\partial\Omega.$ Consider the general diffusion equations in divergence form:
\begin{equation}\label{sys0}
 u _t=\dv  \sigma(Du) \quad \mbox{in $\Omega_T=\Omega\times (0,T)$} 
\end{equation}
for $ u =(u^1,\dots,u^m),$ where each $u^i$ is a function of $(x,t)\in \Omega_T$,  $ u _t=(u^1_t,\dots,u_t^m)$ is the time-derivative of $u$, $D u =(u^i_{x_k})$ is  the spatial Jacobi matrix of $u,$ and $\sigma=(\sigma_k^i(A))\colon \M^{m\times  n}\to \M^{m\times n}$ is a given  diffusion flux function; here $\M^{m\times n}$ denotes the space of real  $m\times n$ matrices. 
When $m\ge 2$, equation (\ref{sys0}) is  a system of $m$ quasilinear partial differential equations:
\[
u^i_t=\sum_{k=1}^n (\sigma^i_k(D u ))_{x_k}\quad (i=1,2,\dots,m).
\]
 
 By a weak solution  of diffusion equation (\ref{sys0}) we mean a function $ u \in W^{1,1}_{loc}(\Omega_T; \R^m)$ that satisfies, for each $i=1,\dots,m$,  
\[
\int_{\Omega_T} \Big (u^i\phi_t -\sum_{k=1}^n \sigma_k^i(D u )\phi_{x_k}\Big )dxdt =0\quad \forall\; \phi\in C^\infty_c(\Omega_T).
\]
The  equation (\ref{sys0}) is said to be  {\em strongly  parabolic}  if   
\begin{equation}\label{mono}
\langle \sigma(A+p\otimes \alpha)-\sigma(A),\, p\otimes \alpha\rangle   \ge \nu |p|^2 |\alpha|^2 
\end{equation}
holds for all $A\in \M^{m\times n}, p\in \R^m$ and $\alpha\in \R^n,$ where $\nu >0$ is a constant; here $\langle A,B\rangle$ stands for the inner product in $\M^{m\times n}\cong \R^{mn}$ and $p\otimes \alpha$ denotes  the matrix $(p_i\alpha_k).$ If $m,n\ge 2$, it is well-known that condition (\ref{mono}) is strictly weaker than  the full {\em  monotonicity} condition: 
\begin{equation}\label{mono-0}
\langle \sigma(A+B)-\sigma(A),\, B \rangle   \ge \nu |B|^2\quad \forall\, A,\,B\in \M^{m\times n}.
\end{equation}
 
 In this paper, we investigate the initial-boundary value problem: 
 \begin{equation}\label{ibvp-1}
\begin{cases}  u _t=\dv \sigma (Du ) & \mbox{in $\Omega_T$,}
\\
\hfill u (x,t)=u_0(x) & (x\in \partial \Omega, \, 0<t<T),\\
 \hfill u (x,0)=u_0(x)& (x\in \Omega),
\end{cases}
\end{equation}
where $u_0\colon \bar\Omega\to\R^m$ is a given initial-boundary function. By a weak solution to this problem we mean a weak solution $u\in W^{1,1}(\Omega_T;\R^m)$ of equation (\ref{sys0}) such that the initial-boundary conditions in (\ref{ibvp-1}) are satisfied in the sense of trace.  

If the flux function $\sigma$ satisfies  the full  monotonicity condition (\ref{mono-0})  then problem (\ref{ibvp-1}) can be studied by the  standard monotonicity  method or general 
parabolic  theory  (see  \cite{Br, Li}); in this case,  one easily sees  that (\ref{ibvp-1}) can have at most one weak solution. 
The main purpose of the present paper is to study  problem (\ref{ibvp-1}) for certain general  {\em nonmonotone} flux functions   $\sigma$ that may satisfy the parabolicity condition (\ref{mono}) in the case of $m, n\ge 2.$ Our main results will apply to the special case of  gradient flows for certain polyconvex functionals; this will  be studied in the sequel  paper Yan \cite{Y1}.

Following the work of {Kim \& Yan} \cite{KY1,KY2,KY3} and {Zhang} \cite{Zh1, Zh2}, we reformulate  equation  (\ref{sys0}) as a space-time {\em partial differential inclusion.}  To this end,  we  introduce the  function $ w =[u,( v ^i)]\colon \Omega_T\to \R^m \times (\R^n)^m$ with  space-time Jacobian  matrix given by
\[
\nabla  w=\begin{bmatrix} Du &  u _t\\
(D v ^i) & ( v ^i_t)\end{bmatrix}\in \M^{(m+nm)\times (n+1)};
\]
here $\M^{(m+nm)\times (n+1)}$ denotes the space of matrices $X$ written in the form of
\begin{equation}\label{mX}
X=\begin{bmatrix} A & a\\ (B^i) & (b^i)\end{bmatrix},\quad A\in \M^{m\times n}, \; a\in \R^m,\; B^i\in \M^{n\times n},\; b^i\in \R^n,
\end{equation}
where   $i=1,\dots, m.$  
For  $ z =(z^1,\dots,z^m)\in\R^m$, we define  the  matrix set $\mathcal K( z )$ in  $\M^{(m+nm)\times (n+1)}$  by
\begin{equation}\label{set-K}
\mathcal K( z )=\left\{ \begin{bmatrix} A& a\\ (B^i)& (\sigma^i(A))\end{bmatrix} :  \tr (B^i)=z^i \; (i=1,\dots,m)\right \}.
\end{equation}

We consider  the   (space-time)   {\em nonhomogeneous} partial differential inclusion
\begin{equation}\label{sys2}
\nabla  w(x,t)\in \mathcal K( u (x,t)) \quad a.e.\; (x,t)\in\Omega_T,
\end{equation}
which is equivalent to the first-order differential system
\begin{equation}\label{sys21}
u^i=\dv v^i,\;\; v^i_t=\sigma^i(Du)  \quad a.e.\; (x,t)\in\Omega_T 
\end{equation}
for $i=1,\dots,m.$ Note that any solution $w=[u,(v^i)]$ of inclusion (\ref{sys2}) or system (\ref{sys21}) would produce  a weak solution $u$ of equation (\ref{sys0}).
 
 There are primarily two approaches for studying partial differential inclusions.  One   is the generalization of {Gromov}'s convex integration method (see 
\cite{G1})   by {M\"uller \&  \v Sver\'ak} \cite{MSv, MSv2}; the other is the {Baire} category
method  developed by {Dacorogna \&  Marcellini}  \cite{DM1, DM} based on  the ideas of  ordinary differential inclusions (see the references in \cite{DM}). Both  methods  rely on  certain approximation and reduction schemes; see also Kirchheim \cite{Ki}, M\"uller \& Sychev  \cite{MSy} and Yan \cite{Y0}. Such  schemes have recently been greatly developed and found remarkable success  in the study of many important problems of partial differential equations; see, e.g., 
 the incompressible Euler equation  by {De Lellis \& Sz\'ekelyhidi} \cite{DS}, the porous media equation by {Cordoba,  Faraco \& Gancedo} \cite{CFG}, the active scalar equation by {Shvydkoy} \cite{Sh},   the  2-D Monge-Amp\`ere equations by  {Lewicka \& Pakzad} \cite{LP} and, most recently, the proof of Onsager's conjecture by {Isett} \cite{Is} and the nonuniqueness  of weak solutions of the Navier-Stokes equation by {Buckmaster \&  Vicol} \cite{BV}.

One difficulty concerning our differential inclusion (\ref{sys2}), as in  the single unknown function cases of  \cite{KY1,KY2,KY3,Zh1,Zh2},  is that the sets $\mathcal K(z)$ are not constant (thus the inclusion is called {\em nonhomogeneous}) and their convex hulls contain no interior points in the full matrix space; this prevents one from applying the general existence results of \cite{MSy}. Another difficulty concerning (\ref{sys2}) is that when $m,n\ge 2$ the set $\mathcal K(z)$ has  a trivial {\em lamination hull} under  the strong parabolicity condition (\ref{mono}), which makes  the constructions  based on rank-one connections unfeasible for (\ref{sys2}) in the system case. To overcome this difficulty, we explore  the more relevant  {\em rank-one convex hull} and  {\em $T_N$-configurations} concerning set  $\mathcal K(z)$ (see Section \ref{s-TN}).
We introduce a structural condition on  the diffusion flux function $\sigma$, called {\em Condition (OC)} (Definition \ref{cond-C}); the precise definition of this condition along with further discussions will be  given in Section \ref{section-tau_N}, but  we first state our main general existence theorem under this condition. 

\begin{thm} \label{thm-main-1} Let $\sigma\colon \M^{m\times n}\to \M^{m\times n}$ be continuous and satisfy  the {\em Condition  (OC)} with open set $\Sigma \subset \M^{m\times n}\times (\R^n)^m$  as given in Definition \ref{cond-C}.  Let $\bar  u \in C^1(\bar\Omega_T;\R^m)$ and $\bar v^1,\dots, \bar  v ^m\in C^1(\bar\Omega_T;\R^n)$ satisfy
\begin{equation}\label{suff-0}
\bar u^i=\dv \bar  v ^i \;\; (i=1,\dots,m), \quad [D\bar  u ,\,(\bar  v ^i_t)]\in \Sigma  \quad \mbox{on $\,\bar\Omega_T.$}
\end{equation}
 Then there exists a sequence $ \{u_\alpha\}$ of weak solutions of   diffusion system $(\ref{sys0})$ in $W^{1,\infty}(\Omega_T;\R^m)$  satisfying $u_\alpha |_{\partial \Omega_T}=\bar u$  that converges weakly* to $ \bar  u$ in $W^{1,\infty}(\Omega_T;\R^m)$ as $\alpha\to\infty.$
  \end{thm}

This result follows by combining Theorem \ref{baire-1} and Theorem \ref{density1} to be proved later.
We remark that  condition (\ref{suff-0}) provides  some  {\em relaxation}   for  inclusion  (\ref{sys2}) or system (\ref{sys21}); in this connection, any  function  $\bar u$ as described by  (\ref{suff-0}) will be  called a {\em subsolution} of  diffusion system (\ref{sys0}). 

The following useful result is immediate.

\begin{pro}\label{subsolns} Assume $\Sigma \subset \M^{m\times n}\times (\R^n)^m$ is a nonempty open set.   Let 
$u_0 \in C^1( \bar\Omega;\R^m)$ and $v_0^i,\, f_0^i\in C^1(\bar\Omega;\R^n)$ $(i=1,\dots ,m)$ satisfy 
\begin{equation}\label{suff-1}
u_0^i=\dv v_0^i,\quad \dv f_0 ^i=0,\quad [D u_0,\,(f_0^i)]\in \Sigma  \quad \mbox{on $\,\bar\Omega$.}
\end{equation}
Assume $h^i\in C^2(\bar\Omega;\R^n)$, $g^i=\dv h^i$ $(i=1,\dots,m)$ and let $\bar  u \in C^1(\bar\Omega_T;\R^m)$  and $\bar  v ^i \in C^1(\bar\Omega_T;\R^n)$   
 be defined  by
\begin{equation}\label{suff-2}
 \bar u^i(x,t)=u^i_0(x) + \epsilon g^i(x) t, \quad  \bar v^i(x,t)=v_0^i (x)+ f_0^i (x) t + \epsilon h^i(x) t.
\end{equation}
Then, if $|\epsilon|$ is  sufficiently small, $\bar u$ and $\bar v^i$ satisfy condition $(\ref{suff-0})$. 
\end{pro}

 From Theorem \ref{thm-main-1}, we obtain the  nonuniqueness and instability result concerning problem $(\ref{ibvp-1}).$ 
 
\begin{thm}\label{instability-thm} Let $\sigma\colon \M^{m\times n}\to \M^{m\times n}$ be continuous and satisfy the  {\em Condition  (OC)} with open set $\Sigma \subset \M^{m\times n}\times (\R^n)^m.$    Assume  that $u_0 \in C^1( \bar\Omega;\R^m)$ satisfies the stationary equation  $
\dv \sigma (Du_0)=0$ weakly  in $\Omega$ and that there exist  $v_0^i,\, f_0^i\in C^1(\bar\Omega;\R^n)$ $(i=1,\dots ,m)$ such that  condition $(\ref{suff-1})$ holds.  Then problem $(\ref{ibvp-1})$  possesses a weakly* convergent sequence  of Lipschitz weak solutions whose limit   is not a weak solution itself.
\end{thm}
\begin{proof} Define  $\bar  u$ and $\bar  v^i$ as in  Proposition \ref{subsolns} with  $h^i \in C^\infty_c(\Omega;\R^n)$ and $g^i=\dv  h^i$ satisfying 
$g^i(x)=\dv h^i(x)=1$ on some fixed nonempty open set $\Omega'\subset\subset \Omega$ for all $i=1, \dots,m.$  Then  $\bar u(x,t)$  satisfies  
\[
\bar u (x,t)=u_0(x) \;\;  (x\in \partial \Omega, \, 0<t<T),\quad \bar u (x,0)=u_0(x)\;\; (x\in \Omega).
\]
Also, for all sufficiently small $|\epsilon|> 0$, condition (\ref{suff-0}) is satisfied. 
 Thus, each solution $u_\alpha$ as determined in  Theorem \ref{thm-main-1} solves problem (\ref{ibvp-1}), but  the weak* limit $\bar u$ is not a weak solution of   (\ref{sys0}) because   
\[
\epsilon g(x)\ne \dv \sigma(Du_0(x) + \epsilon t Dg(x)) \quad \mbox{in $\Omega_T,$} 
\]
as $u_0$ is a stationary solution and $g(x)=1$ on $\Omega'.$
 \end{proof}
 
 \begin{remk} Since the weak* limit is not a weak solution, the Lipschitz weak solutions  in the  sequence of the theorem will eventually be all distinct and  different from the stationary solution $u_0(x).$ This shows that under the Condition (OC) the initial-boundary value problem (\ref{ibvp-1}) is highly {\em ill-posed}; this turns out to be the case even  for certain strongly polyconvex gradient flows, as will be proved in \cite{Y1}.
 \end{remk}
 
 \begin{cor}\label{main-cor-0} Let $\sigma\colon \M^{m\times n}\to \M^{m\times n}$ be continuous and satisfy  the {\em Condition  (OC)} with open set $\Sigma \subset \M^{m\times n}\times (\R^n)^m.$    Assume $[A,(b^i)] \in \Sigma.$ Then the initial-boundary value problem
 \begin{equation}\label{ibvp-g}
 \begin{cases}  u _t=\dv \, \sigma(D u ) & \mbox{in $\Omega_T$,}
\\
\hfill u (x,t)=Ax & (x\in \partial \Omega,\, 0<t<T),\\
 \hfill u (x,0)=Ax & (x\in \Omega),
\end{cases}
\end{equation} 
possesses a weakly* convergent sequence  of Lipschitz weak solutions whose   limit   is not a weak solution itself.
\end{cor} 
 \begin{proof} The result follows easily from Theorem \ref{instability-thm} by choosing $u_0(x)=Ax$, $f_0^i(x)=(b_1^i,\dots,b^i_n)$ and $v_0^i(x)= \frac12 (a_{i1}x_1^2,\dots,a_{in}x^2_n).$
 \end{proof}
  
\section{Rank-one convex hull and admissible $T_N$-configurations}\label{s-TN}

\subsection{Rank-one  hull and $T_N$-configurations} We review some useful definitions and refer to {Dacorogna} \cite{D}  for further related results and details.  
 
 Let $p,q\in\mathbb N.$  A function $g\colon \M^{p\times q}\to \R$ is said to be {\em rank-one convex} if for all given $X,Y$ in $\M^{p\times q}$ with $\rank Y=1$ the function $h(t)=g(X+tY)$ is convex in $t\in \R.$ 

For a compact set $S$ in $\M^{p\times q}$, we define the {\em rank-one convex hull} of $S$  by
\[
S^{rc}=\cap\{g^{-1}(0): \mbox{$g\ge 0$ is rank-one convex and $g|_{S}=0$}\},
\]
and, for general sets $E\subset \M^{p\times q}$, we define
\[
E^{rc}=\cup\{S^{rc}: \mbox{$S\subset E$ and is compact}\}.
\]
Recall that the {\em convex hull} of  set $E\subset \M^{p\times q}$ is (equivalently) defined by
\[
E^c=\{\sum_{j=1}^{pq+1}\lambda_jX_j:  0\le\lambda_j\le 1,\; \sum_{j=1}^{pq+1} \lambda_j=1,\; X_j\in E\}.
\]
Then $E^{rc}\subseteq E^c.$ 
Now let  $E^{lc,0}=E$  and define inductively, 
\[
E^{lc,i+1}=E^{lc,i}\cup \{\lambda X+(1-\lambda)Y : \lambda\in (0,1),\, X,Y\in E^{lc,i},\, \rank(X-Y)= 1\}
\]
for $i=0,1,\dots;$ that is, $E^{lc,i+1}=(E^{lc,i})^{lc,1}.$ The {\em lamination  hull} of $E$ is then defined by
\[
E^{lc}=\cup_{i=0}^\infty E^{lc,i}.
\]
One  has  that $E^{lc}\subseteq E^{rc}\subseteq  E^c.$  The famous  {\em $T_4$-configuration} of Tartar \cite{T} shows a set $E$ may contain  no rank-one connections (i.e., $\rank(X-Y)\ne 1$ for all $X\ne Y$ in $E$) and thus $E^{lc}=E,$  but its rank-one convex hull $E^{rc}$ can be  strictly larger than $E.$ 
 
We have the following general definition; see, e.g.,  \cite{KMS}.

\begin{defn}[$T_N$-configuration]\label{T_N} Let  $N\ge 2$ and $ X_1, \dots,X_N\in \M^{p\times q}.$  We say that the  $N$-tuple $(X_1,X_2,\dots,X_N)$  is a {\em $T_N$-configuration} if  there exist matrices $P, \, C_1, \dots,C_N$ in $\M^{p\times q}$ and real numbers $\kappa_1,\dots,\kappa_N$, with  $\rank(C_j)=1,\, 
 \sum_{j=1}^N C_j=0$ and $\kappa_j>1,$ such that 
 \begin{equation}\label{TN}
\begin{cases}
X_1   =   P+\kappa_1C_1,\\
X_2   =  P+C_1+\kappa_2C_2,\\
 \quad  \vdots   \\
X_N   =  P+C_1+\dots +C_{N-1}+\kappa_NC_N.
\end{cases}
\end{equation}
In this case,  define $P_1=P, \; P_j=P+C_1+\dots +C_{j-1}$ for $j=2,\dots,N$, and
\begin{equation}\label{TTN}
 T (X_1,\dots,X_N)=\cup_{j=1}^N (X_j,P_j),
\end{equation}
where  $(X_j,\,P_j)=\{(1-\lambda) X_j+ \lambda P_j\colon 0<\lambda < 1\}$.  Let $\bar{ T }(X_1,\dots,X_N)=\cup_{j=1}^N [X_j,P_j]$  be the closure of $ T (X_1,\dots,X_N).$  
\end{defn}

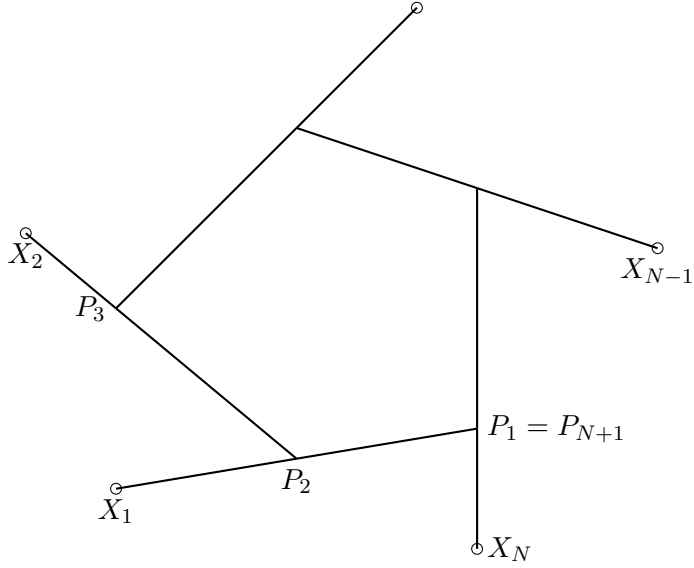
\begin{figure}[ht]
\begin{center}
\begin{tikzpicture}[scale =0.8]
\draw[thick] (-5,-2)--(1,-1);
  \draw(-5,-2) node[below]{$X_1$};
   \draw(1,-1) node[right]{$P_1=P_{N+1}$};
     \draw[thick] (1,-1)--(1,3);
  \draw[thick] (1,-1)--(1,-3);
  \draw[thick] (-2,-1.5)--(-5,1);
   \draw[thick] (-6.5,2.25)--(-5,1);
   \draw[thick] (-5,1)--(-2,4);
     \draw[thick] (0,6)--(-2,4);
    \draw[thick] (-2,4)--(1,3);
       \draw[thick] (1,3)--(4,2);
     \draw(1,-3) node[right]{$X_{N}$};
        \draw(-5,1) node[left]{$P_3$};
         \draw(-2,-1.5) node[below]{$P_2$};
    \draw(4,2) node[below]{$X_{N-1}$};
   \draw(-6.5,2.25) node[below]{$X_2$};
       \draw (-5,-2) circle (0.09);
        \draw(4,2) circle (0.09);
         \draw(1,-3) circle (0.09);
          \draw(-6.5,2.25) circle (0.09);
           \draw(0,6) circle (0.09);
\end{tikzpicture}
\end{center}
\caption{A general $T_N$-configuration $(X_1,X_2,\dots,X_N)$ determined by $P=P_1=P_{N+1}$ and $C_j=P_{j+1}-P_j$ for $j=1,\dots,N.$} 
\label{fig1}
\end{figure}

\begin{remk}   Here, we allow $N=2$ and  the set $\{X_1,X_2,\dots,X_N\}$ to contain   rank-one connections.  With  $P_{N+1}=P_1$, one has 
$P_{j+1}=\frac{1}{\kappa_j}X_j+(1-\frac{1}{\kappa_j})P_j$ for all $j=1,\dots,N;$ hence, there exist  numbers $\nu_i^j\in (0,1)$ for $i,j=1,2,\dots,N$ such that 
\begin{equation}\label{P-j}
\sum_{i=1}^N \nu_i^j=1,\quad P_j=\sum_{i=1}^N \nu_i^j X_i \quad (j=1,2,\dots,N).
\end{equation}
For example, when $j=1$, if $\lambda_i=1-\frac{1}{\kappa_i}$ then
\begin{equation}\label{nu_i^1}
\nu_i^1=  \frac{(1- \lambda_i){\lambda_{i+1} \cdots  {\lambda_N}}}{1-\lambda_1\cdots\lambda_N}  \quad   \forall\, i=1,\dots,N-1; \; \; \nu_N^1=\frac{1-\lambda_N}{1-\lambda_1\cdots\lambda_N}.
\end{equation}
Consequently, each $Y\in T(X_1,\dots,X_N)$ can be written as a specific convex combination of $X_i$'s: $Y=\sum_{i=1}^N \nu_i X_i,$ with $\nu_i$'s being the specific numbers determined by the numbers $\nu_i^j$ in  (\ref{P-j}).
\end{remk}
 
\begin{lem} Let $(X_1,X_2,\dots,X_N)$ be a $T_N$-configuration in $\M^{p\times q}$ given by $(\ref{TN})$. Then
$
\bar{ T }(X_1,\dots,X_N)\subset \{X_1,\dots,X_N\}^{rc}.$ Moreover, if $1<\kappa_j'<\kappa_j$  and $X_j'=P_j+\kappa_j' C_j$ for  $j=1,\dots,N$, then the $N$-tuple $(X'_1,X_2',\dots,X_N')$ is also  a $T_N$ configuration, and $
\bar{ T }(X'_1,\dots,X_N')\subset  T (X_1,\dots,X_N).$ 
\end{lem}

\begin{proof}  Let $g\ge 0$ be  any rank-one convex function on $\M^{p\times q}$ such that $g(X_j)=0$  for all $j=1,\dots,N.$ Then, with $P_{N+1}=P_1$,
\[
g(P_{j+1})\le \frac{1}{\kappa_j}g(X_j)+(1-\frac{1}{\kappa_j})g(P_j)=(1-\frac{1}{\kappa_j})g(P_j)\quad \forall\, j=1,\dots,N,
\]
and thus 
\[
\sum_{j=1}^N g(P_{j}) =\sum_{j=1}^N g(P_{j+1}) \le \sum_{j=1}^N (1-\frac{1}{\kappa_j})g(P_j),
\]
from which it follows that $g(P_j)=0$ for all $j=1,\dots,N.$ Hence, \[
0\le g(\lambda X_j+(1-\lambda)P_j)\le\lambda g(X_j)+(1-\lambda)g(P_j)=0\]
 for all $0\le \lambda\le 1,$ which proves $[X_j,\,P_j]\subseteq g^{-1}(0)$ and thus $\bar{ T }(X_1,\dots,X_N)\subset \{X_1,\dots,X_N\}^{rc}.$ 
 Clearly,  $(X_1',\dots,X_N')$ is a $T_N$-configuration determined by  $P, C_j$ and  $\kappa_j'>1,$ and, since $\kappa_j'<\kappa_j$, it is obvious that $\bar{ T }(X'_1,\dots,X_N')\subset  T (X_1,\dots,X_N).$ 
\end{proof}

\subsection{Admissible $T_N$-configurations in $\M^{(m+nm)\times (n+1)}$}   To  study  the space-time   inclusion (\ref{sys2}), we focus on  the $T_N$-configurations in $\M^{(m+nm)\times (n+1)}.$

\begin{pro}\label{pro1} Let $\sigma \colon \M^{m\times n}\to\M^{m\times n}$ and  define
\begin{equation}\label{set-K0}
\mathcal K =\mathcal K(0)=\left\{ \begin{bmatrix} A& a\\ (B^i)& (\sigma^i(A))\end{bmatrix} :  \tr (B^i)=0 \; (i=1,\dots,m)\right \}.
\end{equation}  
Assume $X=\begin{bmatrix} A & a\\ (B^i) & (b^i)\end{bmatrix}\in  \mathcal K^c;$ then $\tr (B^i) =0$ for all $i=1,\dots,m.$ Moreover, if $\sigma$ satisfies the strong parabolicity condition $(\ref{mono}),$ then    
$\mathcal K^{lc}=\mathcal K.$ 
\end{pro}

\begin{proof} The conclusion about the convex hull $\mathcal K^c$ follows easily.
To prove the second part, let $Y, Y+C$ be in $\mathcal K,$ with $\rank C=1$. Assume 
\[
Y=\begin{bmatrix} A& a\\ (B^i) & (\sigma^i(A))\end{bmatrix},\quad C=\begin{bmatrix} p\otimes \alpha & sp\\ (\beta^i\otimes \alpha) & (s\beta^i)\end{bmatrix},
\]
where $\tr  (B^i)=0$ and $(|p|+\sum_{i}|\beta^i|)(|\alpha|+|s|)\ne 0.$ From $Y+C\in\mathcal K$  one has
\begin{equation}\label{pf-pro1}
\beta^i\cdot \alpha=0,\quad \sigma^i(A+p\otimes \alpha)=\sigma^i(A)+s\beta^i \;\; (i=1,\dots,m);
\end{equation}
hence $(\sigma^i(A+p\otimes \alpha)-\sigma^i(A))\cdot \alpha=0$ for  $i=1,\dots,m.$ So, by  (\ref{mono}),
\[
0= \sum_{i=1}^m p_i(\sigma^i(A+p\otimes \alpha)-\sigma^i(A))\cdot \alpha =\langle \sigma(A+p\otimes \alpha)-\sigma(A), \,p\otimes \alpha\rangle \ge \nu |p|^2|\alpha|^2.
\]
Thus $|p||\alpha|=0$, and so $p\otimes \alpha=0,$ which by (\ref{pf-pro1}) gives  $s\beta^i=0$ and $C=\begin{bmatrix} 0 & sp\\ (\beta^i\otimes \alpha)& (0)\end{bmatrix}.$ This proves $(Y,Y+C)\subset \mathcal K,$ and so   $\mathcal K^{lc,1}=\mathcal K;$ thus, by induction,  $\mathcal K^{lc,j}=\mathcal K$ for all $j$, which proves   $\mathcal K^{lc}=\mathcal K.$
\end{proof}

Suppose  that $(X_1,\dots,X_N)$ is a $T_N$-configuration determined by $P, C_j$ and $\kappa_j$ and that  $X_j\in \mathcal K$ for each $j$. Then $P_j\in \mathcal K^{rc}\subset \mathcal K^{c}$ for $j=1,\dots,N;$ thus  the rank-one matrices 
\[
C_j= P_{j+1}-P_j =\begin{bmatrix} p_j\otimes \alpha_j & s_jp_j\\ (\beta_j^i\otimes \alpha_j)& (s_j\beta_j^i)\end{bmatrix}\;\; (j=1,\dots,N)
\]
must  satisfy  the  {\em orthogonality} condition:
\begin{equation}\label{orth}
\beta_j^i\cdot \alpha_j=0 \quad (i=1,\dots,m;\; j=1,\dots,N).
\end{equation}

\begin{defn}\label{add-s-t-TN}  A rank-one matrix $C$ in $ \M^{(m+nm)\times (n+1)}$ is called {\em admissible} if it is of the form
\begin{equation}\label{C0}
C=\begin{bmatrix} p\otimes \alpha & sp\\ (\beta^i\otimes \alpha)& (s\beta^i)\end{bmatrix} \quad\mbox{with $ \alpha\ne 0,\; \; \beta^i\cdot \alpha=0 \;\; (i=1,\dots,m).$}
\end{equation}
 A  $T_N$-configuration in 
$\M^{(m+nm)\times (n+1)}$ is called  {\em admissible}   if  all its determining rank-one matrices $C_j$'s  are admissible. 
\end{defn}

\begin{lem}\label{lem1} Let $0<\lambda<1$ and
$C=\begin{bmatrix} p\otimes \alpha & sp\\ (\beta^i\otimes \alpha)& (s\beta^i)\end{bmatrix}$ be admissible. Then, for any bounded open  set $G\subset \R^{n+1}$ and $0<\epsilon<1,$   there exists a function  $\omega=[\varphi, (\psi^i)]\in C^\infty_c(\R^{n+1};\R^m\times (\R^n)^m)$ such that
\begin{itemize}
\item[(a)]  $\supp \varphi \subset \subset G,\;\supp\psi^i \subset\subset  G \, (i=1,\dots,m);$
\item[(b)] $\dv\psi^i=0$ in $\R^{n+1}$ for all $i=1,\dots,m;$
\item[(c)] $\int_{\R^n} \varphi(x,t)\,dx=0$ for all $t\in\R;$
\item[(d)] $\|\omega\|_{L^\infty(\R^{n+1})}<\epsilon$ and $\nabla \omega \in [\lambda C,\,(\lambda-1)C]_\epsilon$ on $\R^{n+1};$
\item[(e)] there exist open  sets $G', G''\subset\subset G$ such that
\[
\begin{cases}
\mbox{$\nabla\omega(x,t)=\lambda C$ in $G',$} &\mbox{ $|G'| \ge (1-\epsilon)(1-\lambda)|G|,$}
\\
\mbox{$\nabla\omega(x,t)=(\lambda-1)C$ in $G'',$} &\mbox{ $|G''| \ge (1-\epsilon)\lambda |G|.$}
\end{cases}
\]
\end{itemize}
\end{lem}
\begin{proof} The result with  $m=1$  has been proved in \cite[Lemma 4.5]{KY1}. For general $m$, following the same idea, given $h\in C^\infty(\R^{n+1}),$ let
\[
\mathcal P[h]=[ u ,( v ^i)]=[(\alpha\cdot Dh)p,\, (\beta^i\otimes \alpha-\alpha\otimes \beta^i)Dh].
\]
Then $\dv  v ^i=0$. The proof can be achieved by setting $\omega=[\varphi, (\psi^i)]=\mathcal P[h]$, where $h=\zeta(x,t)f(\alpha\cdot x+st)$ for some suitably chosen $\zeta\in C_c^\infty(\R^{n+1})$ with $\supp\zeta\subset\subset G$ and highly oscillatory $f\in C^\infty(\R)$ with $|\alpha|^2 f''$ taking mostly the values of $0, \lambda, \lambda-1$ and with $f$ and $f'$ being sufficiently small.
\end{proof}

Thanks to the specific  structure  of   admissible  $T_N$-configurations, we  establish the following  result  directly by iterating Lemma \ref{lem1} in finite steps, without using the rank-one convex Young measures as in \cite{MSv2}.

\begin{thm}\label{thm1} Assume that $(X_1, \dots,X_N)$ is an admissible  $T_N$-configuration in   $\M^{(m+nm)\times (n+1)}.$  Let $Y\in  T (X_1,\dots,X_N)$ have the specific convex combination form $Y=\sum_{i=1}^N \nu_i X_i$, $G$ be  bounded  open  in $\R^{n+1},$  and $0<\epsilon<1.$  Then there exists  $\omega=[\varphi, (\psi^i)]\in C^\infty_c(\R^{n+1};\R^m\times (\R^n)^m)$ such that
\begin{itemize}
\item[(a)]  $\supp \varphi \subset\subset G,\;\supp\psi^i \subset\subset G \, (i=1,\dots,m);$
\item[(b)] $\dv\psi^i=0$ in $\R^{n+1}$ for all $i=1,\dots,m;$
\item[(c)] $\int_{\R^n} \varphi(x,t)\,dx=0$ for all $t\in\R;$
\item[(d)] $\|\omega\|_{L^\infty(\R^{n+1})}<\epsilon$ and  $Y+\nabla \omega \in [\bar{ T }(X_1,\dots,X_N)]_\epsilon$ on $\R^{n+1};$  
\item[(e)] there exist  disjoint open  sets $V_i\subset\subset G$ such that 
$
|V_i|\ge (1-\epsilon) \nu_i |G|$ and $ Y+\nabla\omega(x,t) =X_i$ in $V_i;$ in particular, if $V=\cup_{i=1}^N V_i,$ then $|V|\ge (1-\epsilon)|G|$ and $Y+\nabla\omega(x,t)\in\{X_1,\dots,X_N\}$ in $V.$
\end{itemize}
\end{thm}

\begin{proof} Let $(X_1, \dots,X_N)$ be determined by $C_j, P_j$ and $\kappa_j.$ 
Without loss of generality, we assume $Y= (1-\lambda) X_1+\lambda P_1$ for some $0<\lambda<1.$ 
Let $C=X_1-P_1=\kappa_1C_1.$  Since $C$ is admissible, applying Lemma \ref{lem1}, we obtain  $\omega_0 \in C_c^\infty(\R^{n+1})$ with $\supp \omega_0 \subset\subset G=G_0$ satisfying (a)-(e) of the lemma, with a number $\epsilon'\in (0,\epsilon)$ to be determined later. Hence $Y+\nabla \omega_0 =Y+\lambda C=X_1$ in $G_0'$ and $Y+\nabla \omega_0 =Y+(\lambda-1)C=P_1$ in $G_0'',$ where
\[
|G_0'|\ge (1-\epsilon')(1-\lambda)|G|, \;\; |G_0''|\ge (1-\epsilon')\lambda |G|.
\]
We  next modify the function $\omega_0$ on the open set $G_1\equiv G_0'',$ where $Y+\nabla \omega_0=P_1$.
For simplicity, set $\lambda_j= 1-\frac{1}{\kappa_j}\in (0,1)$ for $j=1,\dots,N.$  Since $P_1=P_{N+1}=(1-\lambda_N) X_N+ \lambda_N P_N$, using $C=X_N-P_N=\kappa_NC_N$ and applying Lemma \ref{lem1} to the open set $G_1$ and the matrix $P_1$,  we obtain $\omega_1 \in C_c^\infty(\R^{n+1})$ with $\supp \omega_1\subset\subset G_1$ satisfying (a)-(e) of the lemma with open set $G_1$ and number $\epsilon'$. Hence $P_1+\nabla \omega_1 =X_N$ in $G_1'\subset\subset G_1$ and $P_1+\nabla \omega_1  =P_N$ in $G_1''\subset\subset G_1,$ where
\[
|G_1'|\ge (1-\epsilon')(1-{\lambda_N})|G_1|,\quad  |G_1''|\ge (1-\epsilon'){\lambda_N} |G_1|.
\]
Let $G_2=G_1''.$ Using $P_N={(1-\lambda_{N-1})}X_{N-1}+ \lambda_{N-1} P_{N-1}$, we repeat the above procedure to obtain $\omega_2 \in C_c^\infty(\R^{n+1})$ with $\supp \omega_2\subset\subset G_2$ satisfying (a)-(e) of the lemma with open set $G_2$ and number $\epsilon'.$  
Hence $P_N+\nabla \omega_2 =X_{N-1}$ in $G_2'\subset\subset G_2$ and $P_N+\nabla \omega_2  =P_{N-1}$ in $G_2''\subset\subset G_2,$ where
\[
|G_2'|\ge (1-\epsilon')(1-{\lambda_{N-1}})|G_2| \ge (1-\epsilon')^2 (1-{\lambda_{N-1}})  {\lambda_N}   |G_1|, 
\]
\[
|G_2''|\ge (1-\epsilon'){\lambda_{N-1}} |G_2|\ge (1-\epsilon')^2 {\lambda_{N-1}} {\lambda_N}   |G_1|.
\]
Continue this procedure until  we reach the identity $P_2={(1-\lambda_{1})}X_{1}+ {\lambda_{1}} P_{1}$ to  obtain $\omega_N \in C_c^\infty(\R^{n+1})$ with $\supp \omega_N\subset\subset G_N=G_{N-1}''$ satisfying (a)-(e) of the lemma with open set $G_N$ and number $\epsilon'.$ Hence $P_2+\nabla \omega_N =X_{1}$ in $G_N'\subset\subset G_N$ and $P_2+\nabla \omega_N  =P_{1}$ in $G_N''\subset\subset G_N,$ with
\[
|G_N'|\ge  (1-\epsilon')^{N}(1-{\lambda_1}) {\lambda_{2}}   \cdots   {\lambda_N}  |G_1|, \quad 
 |G_N''|\ge  (1-\epsilon')^{N} {\lambda_{1}} \cdots   {\lambda_N} |G_1|.
\]
Let $\mu={\lambda_{1}} \cdots   {\lambda_N}$ and $\tilde \omega_1=\omega_1+\cdots +\omega_N.$ Then $\tilde\omega\in C_c^\infty(\R^{n+1})$,   $\supp\tilde \omega_1 \subset\subset G_1$, $Y+\nabla\tilde \omega_1=P_1$ on $G_N'',$ and $Y+\nabla\tilde\omega_1=X_{j}$ on open set $V^1_{j}=G_{N+1-j}'\subset\subset G_1$ for $j=1,\dots,N,$ where
\[
|V^1_{j}|=|G_{N+1-j}'| \ge (1-\epsilon')^{N+1-j} (1-{\lambda_{j}}){\lambda_{j+1}} \cdots  {\lambda_N}  |G_1|.
\]
By (\ref{nu_i^1}),  it follows that $(1-{\lambda_{j}}){\lambda_{j+1}} \cdots  {\lambda_N}  =\nu_j^1 (1-\mu)$ and thus
\begin{equation}\label{V_1j-set}
|V^1_{j}|\ge (1-\epsilon')^{N} \nu_j^1 (1-\mu)|G_1|.
\end{equation}
We  now perform  the previous modification procedures  for the function $\tilde\omega_1$ on the open set $G_{11}\equiv G_N''$ (where  $Y+\nabla\tilde\omega_1=P_1$). In this way, we obtain  open sets $V^2_{j}\subset\subset G_{11}$ for $j=1,\dots,N$  and a function $\tilde \omega_{2}\in C_c^\infty(\R^{n+1})$ with $\supp \tilde\omega_{2}\subset\subset G_{11}$ such that $Y+\nabla\tilde\omega_{2}=X_j$ on $V^2_{j}$ and
\[
|V^2_{j}|\ge (1-\epsilon')^N \nu_j^1(1-\mu)  |G_{11}|\ge (1-\epsilon')^{2N} \nu_j^1 (1-\mu)\mu|G_1|.
\]
Moreover, $Y+\nabla\tilde\omega_{2}=P_1$ on an open set $G_{12}'' \subset\subset G_{11}$ with 
\[
|G_{12}''|\ge  (1-\epsilon')^N \mu |G_{11}|\ge (1-\epsilon')^{2N} \mu^2|G_1|.
\]
We iterate  this procedure  on the open set $G_{21}=G_{12}''$ and,  after $k$ iterations, we obtain  open sets $V^{k+1}_j \subset\subset G_{k1}=G_{(k-1)2}''$ and a function $\tilde \omega_{k+1}\in C_c^\infty(\R^{n+1})$ with $\supp \tilde\omega_{k+1}\subset\subset G_{k1}$ such that $Y+\nabla\tilde\omega_{k+1}=X_j$ on $V^{k+1}_j$ and
\[
|V^{k+1}_j|\ge   (1-\epsilon')^{(k+1)N} \nu_j^1 (1-\mu)\mu^ {k}|G_1|.
\]
In general, $|V^{i+1}_j|\ge   (1-\epsilon')^{(i+1)N} \nu_j^1 (1-\mu)\mu^ {i}|G_1| \ge   (1-\epsilon')^{(k+1)N} \nu_j^1 (1-\mu)\mu^ {i}|G_1|$ for $i=0,1,\dots,k.$ Let  
\[
\tilde\omega=\sum_{j=1}^{k+1} \tilde\omega_{j},\quad \tilde V_j =\bigcup_{i=0}^{k} V_j^{i+1}.
\]
  Then, $\tilde\omega\in C_c^\infty(\R^{n+1})$ with $\supp \tilde\omega \subset\subset G_1$,  $Y+\nabla\tilde\omega=X_j$ in $\tilde V_j \subset\subset G_1$, and 
\[
|\tilde V_j|\ge (1-\epsilon')^{(k+1)N} \nu_j^1 (1-\mu) (\sum_{i=0}^{k} \mu^{i}) |G_1|=(1-\epsilon')^{(k+1)N} \nu_j^1 (1-\mu^{k+1})|G_1|.
\]
Finally, we first select  $k\in\mathbb N$ sufficiently large so   that $(1-\mu^{k+1})\ge (1-\epsilon)^{1/2}$ and then select $0<\epsilon'<\epsilon$  sufficiently small so that
\[
(k+1)(N+1)\epsilon'<\epsilon,\quad (1-\epsilon')^{(k+2)N} \ge (1-\epsilon)^{1/2}.
\]
Let $\omega=\omega_0+\tilde\omega\in C^\infty(\R^{n+1}).$ Then,  $\supp \omega\subset\subset G$,   
$Y+\nabla\omega=X_j$ in $\tilde V_j$ for $j=2,\dots,N$ and $Y+\nabla\omega=X_1$ on $V_1=\tilde V_1\cup G_0'$. 
For $j=2,\dots,N$,
\[
|\tilde V_j|\ge  (1-\epsilon')^{(k+1)N} \nu_j^1 (1-\mu^{k+1})|G_1|\ge  (1-\epsilon) \nu_j^1  \lambda |G|=
 (1-\epsilon)\nu_j|G|.
 \]
Also, noting that $\nu_1=(1-\lambda)+\lambda \nu_1^1$, one has
\[
|V_1|=|G_0'|+|\tilde V_1| \ge  (1-\epsilon')(1-\lambda)|G| + (1-\epsilon) \nu_1^1  \lambda |G| \ge (1-\epsilon)\nu_1 |G|.
\]
 Furthermore, $\|\omega\|_{L^\infty}<(k+1)(N+1)\epsilon' <\epsilon$ and 
\[
\dist(Y+\nabla\omega; \bar{ T }(X_1,\dots,X_N)) <(k+1)(N+1)\epsilon'<\epsilon.
\]
The proof is thus  finished.
\end{proof}

\section{ $\tau_N$-configurations and  the Condition (OC)} \label{section-tau_N}

 Let $\mathbb P$ be the ``diagonal" projection  in $\M^{(m+nm)\times (n+1)}$ defined by
\begin{equation}\label{P}
\mathbb P\left (\begin{bmatrix} A & a\\ (B^i)& (b^i)\end{bmatrix}\right)=[A, (b^i)]\in \M^{m\times n}\times (\R^n)^m.
\end{equation}
Assume $\sigma\colon\M^{m\times n}\to \M^{m\times n}\cong (\R^n)^m$  and  $\mathcal K$ is defined  by (\ref{set-K0}). Then  
  \begin{equation}\label{g-K}
\mathbb K=\mathbb P(\mathcal K)=\{[A,(\sigma^i(A))]:A\in \M^{m\times n}\} 
\end{equation}
is exactly  the graph of $\sigma$ in $\M^{m\times n}\times (\R^n)^m.$  
The following  result is useful.

\begin{lem}\label{lem-X-s}  Let $(X_1,\dots,X_N)$ be an admissible $T_N$-configuration  given by $P, C_j$ and $\kappa_j>1$, where
\[
P=\begin{bmatrix} \tilde A& \tilde a\\ (\tilde B^i)& (\tilde b^i)\end{bmatrix},\quad C_j= \begin{bmatrix} p_j\otimes \alpha_j & s_jp_j\\ (\beta_j^i\otimes \alpha_j)& (s_j\beta_j^i)\end{bmatrix}.
\]
Let 
\begin{equation}\label{C-s}
\tilde P=\begin{bmatrix} \tilde A& 0\\ (0)& (\tilde b^i)\end{bmatrix},\quad C^s_j= \begin{bmatrix} p_j\otimes \alpha_j & ss_jp_j\\ \frac{1}{s}(\beta_j^i\otimes \alpha_j)& (s_j\beta_j^i)\end{bmatrix} \;\; (s\ne 0)
\end{equation}
and let $\tilde P_1^s=\tilde P,$ and $\tilde P^s_j=\tilde P+C_1^s+\cdots +C_{j-1}^s$ for $j=2,\dots,N.$ Then $\{\tilde X_j^s=\tilde P_j^s +\kappa_j C^s_j\}_{j=1}^N$ forms an admissible   $T_N$-configuration  given by $\tilde P, C^s_j$ and $\kappa_j$, with
$
\mathbb P( T (\tilde X^s_1,\dots,\tilde X^s_N))=\mathbb P( T (X_1,\dots,X_N)).
$
Furthermore,  if $X_j\in \mathcal K$ then $\tilde X^s_j\in \mathcal K.$ 
\end{lem}
\begin{proof} Clearly $C_j^s$ is an admissible  rank-one matrix. From $\sum_{j=1}^N C_j=0$, we have $\sum_{j=1}^N C_j^s=0$, and thus the $N$-tuple $(\tilde X^s_1,\dots,\tilde X^s_N)$ forms an admissible $T_N$-configuration given by $\tilde P, C^s_j$ and $\kappa_j.$  Since $\mathbb P(\tilde P)=\mathbb P(P)$ and $\mathbb P(C_j^s)=\mathbb P(C_j),$ one has $\mathbb P(\tilde X_j^s)=\mathbb P(X_j)$ and $\mathbb P(\tilde  P_j^s)=\mathbb P(P_j),$ which proves $\mathbb P((X_j,P_j])=\mathbb P((\tilde X_j^s,\tilde P_j^s])$ and thus \[
\mathbb P( T (\tilde X^s_1,\dots,\tilde X^s_N))=\mathbb P( T (X_1,\dots,X_N)).\]
Finally, if $X_j\in \mathcal K$ then  $\mathbb P(\tilde X_j^s)=\mathbb P(X_j)\in \mathbb K;$ thus $\tilde X^s_j\in \mathcal K.$ 
\end{proof}

 \begin{defn}[$\tau_N$-configuration]\label{def-tauN} We say that an  $N$-tuple $(\xi_1,\xi_2,\dots,\xi_N)$  with  $\xi_j\in \M^{m\times n}\times (\R^n)^m$  is a {\em $\tau_N$-configuration} provided that  there exist  $\rho, \, \gamma_1, \dots,\gamma_N$ in $\M^{m\times n}\times (\R^n)^m,$ and  $\kappa_1>1, \dots,\kappa_N>1$ such that  
\begin{equation}\label{tauN}
\begin{cases}
\xi_1  =   \rho+\kappa_1\gamma_1,\\
\xi_2  =  \rho+\gamma_1+\kappa_2\gamma_2,\\
 \qquad  \vdots   \\
\xi_N   =  \rho+\gamma_1+\cdots +\gamma_{N-1}+\kappa_N\gamma_N,
\end{cases}
\end{equation}
where $\gamma_j =[p_j\otimes \alpha_j, (s_j\beta_j^i)],$ with $s_j\in \R, \alpha_j\ne 0,\beta^i_j\in \R^n$ and $p_j\in \R^m$ satisfying
\begin{eqnarray}
&\sum_{j=1}^N s_jp_j=0,\quad  \sum_{j=1}^N s_j\beta_j^i=0  \;\;\; (i=1,\dots,m),& \label{rk-1}\\
& \sum_{j=1}^N p_j\otimes \alpha_j=0,\quad  \sum_{j=1}^N \beta_j^i\otimes \alpha_j=0 \;\;\; (i=1,\dots,m),& \label{rk-2}\\ 
&\beta^i_j\cdot \alpha_j=0  \quad  (j=1,\dots,N; \; i=1,\dots,m).& \label{rk-3}
\end{eqnarray}
In this case,  define  $\rho_1=\rho,\; \rho_j=\rho+\gamma_1+\cdots +\gamma_{j-1}$  for $j=2,\dots,N$, and
\begin{equation}\label{tauN1}
\mathbf {\tau}(\xi_1,\dots,\xi_N)=\cup_{j=1}^N (\xi_j,\rho_j).
\end{equation} 
\end{defn}

\begin{remk} \label{remk-30}  
 Let  $\mathcal M_N$ be the set of all $\tau_N$-configurations.  As a subset of  $[\M^{m\times n}\times (\R^n)^m]^N$, due to possible degeneracy in conditions (\ref{rk-1})--(\ref{rk-3}), $\mathcal M_N$ may not be a sub-manifold even locally near a given $\tau_N$-configuration.  
We  instead consider  certain special $\tau_N$-configurations of  $\mathcal M_N.$ 
 For example, if $s_j=\alpha_j\cdot \delta$ for some fixed $\delta\in\R^n,$ then condition (\ref{rk-1}) follows automatically from (\ref{rk-2}); the choice of $s_j=\alpha_j\cdot\delta$ represents a time scaling that is proportional to the spatial scaling. See Subsection \ref{dim=2}   for more details.
\end{remk}

 We are now in a position to define our Condition (OC).

 \begin{defn}[{Condition (OC)}]  \label{cond-C} A function  $\sigma\colon \M^{m\times n}\to  \M^{m\times n}$ is said to satisfy  Condition (OC) provided that there exists a nonempty  bounded {\em open set}  $\Sigma \subset \M^{m\times n}\times (\R^n)^m$ satisfying the following property:
 \begin{equation}\label{C}
\begin{cases} \forall\; [A,(b^i)]\in \Sigma \;\; \mbox{$\exists \; N\ge 2$ and $\tau_N$-configuration $(\xi_1,\dots,\xi_N)$ } 
\\
\mbox{such that $\xi_j\in \mathbb K$ for all $j$ and  $[A,(b^i)]\in \tau(\xi_1,\dots,\xi_N) \subseteq \Sigma.$}\end{cases}
\end{equation}
 \end{defn}

\subsection{The case of single unknown function} \label{dim-m=1} Assume $m=1.$ In  this  case, any $\tau_2$-configuration $(\xi_1,\xi_2)$ in $\mathbb K$ is given by $\xi_j=[p_j,\sigma(p_j)]$ for $j=1,2$ with $p_1,p_2\in\R^n$ satisfying $G(p_1,p_2)=0$, where 
\[
G(p,q)=(\sigma(p)-\sigma(q))\cdot (p-q)\quad\forall\,p,\,q\in\R^n.
\]
Let
 $\Delta=\{p\in\R^n\,|\,\mbox{$\sigma$ is $C^1$ near $p$}\}.$  For $p\in\Delta$, denote by $\sigma'(p)$ the $n\times n$ matrix $(\partial \sigma^i(p)/\partial p_j)$ ($i,j=1,\dots,n$). For $p,q\in \Delta$, let  
\[
\delta(p,q)= \det \begin{bmatrix} \sigma'(p)-\sigma'(q) & \sigma(p)-\sigma(q)-\sigma'(q)(p-q)\\
(\sigma(p)-\sigma(q)+\sigma'(p)^T(p-q))^T &0\end{bmatrix}.
\]

\begin{pro}\label{R-set}  Let $\Delta^+, \Delta^-$ be nonempty open subsets of $\Delta.$ Suppose $\delta(p,q)\ne 0$ for all $p\in \Delta^+$, $q\in \Delta^-$ and  $G(p_0,q_0)=0$ for some $p_0\in \Delta^+$, $q_0\in \Delta^-.$  Define 
\[
\Sigma =\left\{[p,\beta]\in \R^n\times \R^n \,\Big |\, \begin{aligned} &\mbox{$\exists\, \lambda\in (0,1),\, p_+\in \Delta^+,\, p_-\in \Delta^-,\, G(p_+,p_-)=0$} \\ & \mbox{with $[p,\beta]= \lambda [p_+,\sigma(p_+)] +(1- \lambda)[p_-,\sigma(p_-)]$}\end{aligned} \right \}.
\]
Then
$\Sigma $ is nonempty, open  and satisfies  condition $(\ref{C})$ with $N=2.$  
\end{pro}

\begin{proof} Since $G(p_0,q_0)=0$, it follows easily that $\Sigma \ne \emptyset$; clearly, $\Sigma $  satisfies (\ref{C}).   To show $\Sigma $ is open, let
  $[\bar p,\bar \beta]\in \Sigma.$ Assume 
$[\bar p, \bar \beta] =\bar \lambda [\bar p_+,\sigma(\bar p_+)]+(1-\bar \lambda)[\bar p_-,\sigma(\bar p_-)],$ where $\bar \lambda\in (0,1)$, $ \bar p_+\in \Delta^+$, $\bar p_-\in \Delta^-$ and $G(p_+,p_-)=0$; thus  
$\bar p_-=\frac{\bar p-\bar \lambda \bar p_+}{1-\bar \lambda}.$ 
Let $q(p_+,\lambda,p)=\frac{ p-\lambda p_+}{1- \lambda},$ and define
\[
F(p_+,\lambda,p,\beta)=\begin{pmatrix} \lambda \sigma(p_+)+(1-\lambda) \sigma(q(p_+,\lambda,p))-\beta\\
 (\sigma(p_+)-\beta) \cdot (p_+-p)\end{pmatrix}
\]
for $p_+\in\R^n$, $\lambda\in (0,1)$, $p\in\R^n$ and $\beta\in \R^n.$ Then, 
\[
\frac{\partial q(p_+,\lambda,p)}{\partial p_+}=-\frac{\lambda}{1-\lambda} I_n, \quad  \frac{\partial q(p_+,\lambda,p)}{\partial \lambda}=\frac{p-p_+}{(1-\lambda)^2},
\] 
\begin{equation}\label{zero-cond}
q(\bar p_+,\bar\lambda,\bar p)=\bar p_-,\quad F(\bar p_+, \bar\lambda,\bar p, \bar \beta)=0,
\end{equation}
and $F$  is $C^1$ near $(\bar p_+,\bar\lambda,\bar p,\bar \beta)$ with partial Jacobian matrix
\[
\frac{\partial F (p_+,\lambda,p,\beta)}{\partial(p_+,\lambda)}\]\[=\begin{bmatrix} \lambda(\sigma'(p_+)-\sigma'(q)) & \sigma(p_+)-\sigma(q)-\sigma'(q)(p_+-q)\\(\sigma'(p_+)^T(p_+-p)+\sigma(p_+)-\beta)^T & 0\end{bmatrix},
\]
where $q=q(p_+,\lambda,p).$ 
From (\ref{zero-cond}),   we have
\[
\det \frac{\partial F (\bar p_+,\bar \lambda,\bar p,\bar \beta)}{\partial(p_+,\lambda)} =\bar \lambda^{n}(1-\bar\lambda) \delta(\bar p_+,\bar p_-)\ne 0.
\]
Hence, by the Implicit Function Theorem, there exist neighborhoods $B_\epsilon(\bar p)$, $B_\epsilon(\bar \beta)$, $ B_\epsilon(\bar p_+)\subset \Delta^+$, $ B_\epsilon(\bar \lambda)\subset (0,1)$ and a $C^1$ function 
\[
(p_+,\lambda)=(p_+(p,\beta),\lambda(p,\beta))\colon B_\epsilon(\bar p)\times B_\epsilon(\bar \beta)\to B_\epsilon(\bar p_+)\times B_\epsilon(\bar \lambda)
\]
 such that,  
for all $(p,\beta)\in B_\epsilon(\bar p)\times B_\epsilon(\bar \beta),$ 
  \[
F(p_+(p,\beta),\lambda(p,\beta),p,\beta)=0,\quad \delta(p_+(p,\beta),p_-(p,\beta))\ne 0,
\]
where $
p_-(p,\beta)=q(p_+(p,\beta),\lambda(p,\beta),p)$ is also $C^1$ in the domain. Since $p_-(\bar p,\bar\beta)=\bar p_-\in \Delta^-$, we may choose an even smaller $\epsilon>0$ so that $p_-(p,\beta)\in \Delta^-$ for  all $[p,\beta]\in B_\epsilon(\bar p)\times B_\epsilon(\bar\beta).$
Hence, for  all $[p,\beta]\in B_\epsilon(\bar p)\times B_\epsilon(\bar\beta),$ it follows  that  $\tilde \lambda =\lambda(p,\beta)\in (0,1)$, $\tilde p_+=p_+(p,\beta)\in \Delta^+$,  $\tilde p_-=p_-(p,\beta)\in \Delta^-$,  $G(\tilde p_+ ,\tilde p_- )=0,$ and 
\[
[p,\beta]=\tilde \lambda [\tilde p_+, \sigma(\tilde p_+)]+(1-\tilde \lambda)[\tilde p_-,\sigma(\tilde p_-)];
\]
thus, by the definition of $\Sigma $, we have  $[p,\beta] \in \Sigma $, and so $\Sigma $ is open.
\end{proof}

\begin{remk}\label{remk-m=1} Suppose that $G(p_0,q_0)=0$ for some $p_0\ne q_0$ in $\Delta;$ this  is guaranteed if $\sigma$ is smooth and {nonmonotone}. Note that $\delta(p_0,q_0)$ is a polynomial of $P=\sigma'(p_0)$ and $Q=\sigma'(q_0);$ write this polynomial as $j(P,Q).$ From $G(p_0,q_0)=0$, it follows  that 
\[
j(I,-I)=2^{n-1}[|\sigma(p_0)-\sigma(q_0)|^2+|p_0-q_0|^2]>0;
\]
thus $j(P,Q)$ is not identically zero. Therefore, for any $\epsilon>0$, there exists a smooth function $\tilde\sigma$ such that  $\|\sigma-\tilde\sigma\|_{C^0}<\epsilon$, $\tilde\sigma$  agrees with $\sigma$  at $p_0$ and $q_0$, $|\tilde\sigma'(p)-\sigma'(p)|<\epsilon$ at both $p=p_0$ and $q_0$, and $\tilde\delta(p_0,q_0)\ne 0.$ Thus for such a function $\tilde\sigma$ one can always find sets $\Delta^\pm;$ so,  by Proposition \ref{R-set},   $\tilde\sigma$ satisfies the Condition (OC) with $N=2.$  
Such  ideas will be greatly explored and generalized   in \cite{Y1} for constructing certain polyconvex gradient flows.
\end{remk}

\subsection{The case of dimension 2}\label{dim=2} Let $n=2;$ then condition (\ref{rk-3}) is equivalent to  $(\beta_j^i)^\perp =q_j^i\alpha_j$ for some $q_j^i\in  \R$, 
 where $\beta^\perp=\beta J=(b,-a)\in \R^2$ for all  $\beta=(a,b)\in \R^2.$ 
Define $\mathcal L\colon \M^{m\times 2}\times (\R^2)^m\to \M^{2m\times 2}$ by
\begin{equation}\label{mapL}
\mathcal L([A,(b^i)])=\begin{bmatrix}A\\BJ\end{bmatrix} \quad \forall\, B=(b^i_k)\in \M^{m\times 2}.
\end{equation}
Then $\mathcal L$ is a linear bijection. 
One easily sees that $(\xi_1,\dots,\xi_N)$ is a $\tau_N$-configuration in $\M^{m\times 2}\times (\R^2)^m$ if and only if $(\mathcal L\xi_1,\dots,\mathcal L\xi_N)$ is a  special $T_N$-configuration in $\M^{2m\times 2},$ whose determining rank-one matrices are given by $C_j=\mathcal L\gamma_j =\begin{pmatrix}p_j\\s_jq_j\end{pmatrix}\otimes \alpha_j$ with $p_j\in \R^m, q_j\in\R^m, \alpha_j\in\R^2\setminus\{0\}$ and $s_j\in\R$ satisfying 
\begin{equation}\label{n=2}
\begin{cases} \sum\limits_{j=1}^N p_j\otimes \alpha_j=0,\quad \sum\limits_{j=1}^N s_jq_j\otimes \alpha_j=0,
\\
\sum\limits_{j=1}^N s_jp_j=0,\quad  \sum\limits_{j=1}^N q_j\otimes \alpha_j\otimes \alpha_j=0;
\end{cases}
\end{equation}
here, for general vectors $a\in\R^p,b\in\R^q$ and $c\in\R^r$, $a\otimes b\otimes c$ denotes the tensor $(a^ib^jc^k)$ for $1\le i\le p, \, 1\le j\le q,\, 1\le k\le r.$ 

\begin{remk}\label{remk-33}  The first two conditions in (\ref{n=2}) are the conditions defining  a general $T_N$-configuration in $\M^{2m\times 2}$, while  the last two conditions are extra requirements for producing a $\tau_N$-configuration in $\M^{m\times 2}\times (\R^2)^m.$ Therefore, a general $T_N$-configuration in $\M^{2m\times 2}$ does not produce a $\tau_N$-configuration in $\M^{m\times 2}\times (\R^2)^m,$ only the special ones satisfying (\ref{n=2}) do.
\end{remk}

The  conditions in (\ref{n=2}) may be degenerate; thus, we will  focus  on certain more specific configurations.

\begin{defn}\label{def-m'} Assume $N\ge 3.$ Let $ M'_N$ be the set of  $T_N$-configurations $(X_1,\dots,X_N)$ in $\M^{2m\times 2}$ whose determining rank-one matrices are given by $C_j= \begin{pmatrix}p_j\\(\alpha_j\cdot \delta) q_j\end{pmatrix}\otimes  \alpha_j,$ where $p_j\in\R^m, \,q_j\in\R^m$, $\alpha_j\in \R^2\setminus\{0\}$,  $\delta\in\R^2$,   at least three of $\alpha_j$'s are  {\em mutually noncollinear}, and 
\begin{equation}\label{tn-1}
 \sum_{j=1}^N p_j\otimes \alpha_j= 0,\quad  \sum_{j=1}^N q_j\otimes \alpha_j\otimes \alpha_j=0.
\end{equation}  
(Thus all conditions in  $(\ref{n=2})$ are automatically satisfied with $s_j=\alpha_j\cdot \delta.$)
We   define  $\mathcal M'_N=\mathcal L^{-1}(M_N')$  to be the  set of  {\em special $\tau_N$-configurations} in $\M^{m\times 2}\times (\R^2)^m.$  
\end{defn}

\begin{lem}\label{2d-count} Let $N\ge 3.$  Then $\dim M'_N=(2m+2)N+1-m$ in $(\M^{2m\times 2})^N.$
\end{lem}
\begin{proof} Let $(X_1,\dots,X_N)\in M'_N$ be determined by $P,C_j$ and $\kappa_j$, with $C_j= \begin{pmatrix}p_j\\(\alpha_j\cdot \delta) q_j\end{pmatrix}\otimes  \alpha_j$ as above. Write  $\alpha_j=(x_j,y_j)\in\R^2.$ Then (\ref{tn-1}) is equivalent to 
\begin{equation}\label{eq-pj-qj}
\begin{cases}
\sum_{j=1}^Nx_jp_j=\sum_{j=1}^Ny_jp_j=0,
\\
\sum_{j=1}^Nx_j^2 q_j=\sum_{j=1}^N y_j^2 q_j=\sum_{j=1}^N x_jy_jq_j=0.
\end{cases}
\end{equation}
Note that, for  $a,b,c,x,y,z\in\R$,
\[
\det\begin{pmatrix}a^2&b^2&c^2\\x^2&y^2&z^2\\ax&by&cz\end{pmatrix}=\det\begin{pmatrix}a&b\\x&y\end{pmatrix}\det\begin{pmatrix}b&c\\y&z\end{pmatrix}\det\begin{pmatrix}c&a\\z&x\end{pmatrix}
\]
and thus, since at least three of $\alpha_j$'s are mutually noncollinear, we have
\[
\rank \begin{pmatrix} x_1&x_2&\dots &x_N\\y_1&y_2&\dots &y_N\end{pmatrix}=2,\;\;  \rank \begin{pmatrix} x_1^2&x_2^2&\dots &x_N^2\\y_1^2&y_2^2&\dots &y_N^2\\
x_1y_1&x_2y_2&\dots &x_Ny_N\end{pmatrix}=3.
\]
Consequently, given $\alpha_j$ as described above, from (\ref{eq-pj-qj}), the dimension of solutions $(p_j)$ in $(\R^m)^N$ is $m(N-2)$ and the dimension of solutions $(q_j)$ in $(\R^m)^N$ is $m(N-3).$ Note that  for  $r_j\ne  0$ the change of variables $\alpha_j\to \frac{1}{r_j}\alpha_j, p_j\to r_jp_j$, $q_j\to r_j^2q_j$ does not change the rank-one matrix $C_j$ and that for $s\ne 0$ the change of variables $\delta\to \frac{1}{s}\delta$ and $q_j\to sq_j$ does not change $C_j$; thus, counting $\alpha_j,\delta, p_j$ and $q_j$, modulo $r_j$ and $s$,  the dimension of $C_j$ equals 
\[
2N+2+m(N-2)+m(N-3)-N-1=(2m+1)N+1-5m.
\]
Finally, 
adding the dimensions of $(C_j), (\kappa_j)$ and $P$ proves the result.
\end{proof}

\begin{remk}   Let $m=n=2$, $K=\mathcal L(\mathbb K)$ and $K_N=K\times \dots \times K$ ($N$-times). By Lemma \ref{2d-count},  $\dim(M'_4\cap K_4) \ge \dim M'_4 + \dim K_4 -\dim (\M^{4\times 2})^4 \ge 7$ and  thus 
$
\dim(\tau(\mathcal M'_4\cap \mathbb K_4))\ge 8=\dim (\M^{2\times 2}\times (\R^2)^2).
$
Therefore,  it  is  plausible that an open set $\Sigma$ in $\M^{2\times 2}\times (\R^2)^2$ can be found so that the Condition (OC)  is satisfied  with  $N=4$  for certain functions  $\sigma$ satisfying (\ref{mono}). However, if   $\sigma=DF$ and $F$ is strongly polyconvex, then a result of  \cite[Proposition 3.11]{KMS} shows that it is impossible for any  $\tau_4$-configuration  to be embedded on $\mathbb K_4.$   In Yan \cite{Y1}, motivated by \cite{Sz},  we show that, for certain polyconvex functions $F$, the embedding of even certain special $\tau_5$-configurations in $\mathbb K_5$ is possible.  
  \end{remk}

\section{General existence  by Baire's category}

Given an open set $G\subset \R^p$ and a function $\psi\in W^{1,\infty}(G;\R^q)$, we use the following notation:
\[
W^{1,\infty}_\psi(G;\R^q)=\{\phi \in W^{1,\infty}(G;\R^q):\;\phi|_{\partial G}=\psi\}.
\]

\begin{thm}\label{baire-1} Let $\sigma\colon\M^{m\times n}\to\M^{m\times n}$ be continuous, $\bar  u \in W^{1,\infty}(\Omega_T;\R^m)$, and let $\mathcal U$ be a nonempty bounded set in $W_{\bar  u }^{1,\infty}(\Omega_T;\R^m)$. 
Assume for each $\epsilon>0$ there exists a subset 
\[
\mathcal U_\epsilon\subset \{ u \in \mathcal U\,|\, \| u _t-\dv \sigma(D u )\|_{H^{-1}(\Omega_T)}<\epsilon\}
\]
which is dense in $\mathcal U$ in the $L^\infty(\Omega_T;\R^m)$-norm. Then the solution set 
\[
\mathcal S=\{ u \in W_{\bar  u }^{1,\infty}(\Omega_T;\R^m)\,|\,\mbox{\em $ u $ is Lipschitz solution of (\ref{sys0})}\}  
\]
is also dense (thus nonempty) in $\mathcal U$ in the $L^\infty(\Omega_T;\R^m)$-norm. 
\end{thm}

\begin{proof} Let $\mathcal X$ be the closure of $\mathcal U$ in $L^\infty(\Omega_T;\R^m).$ Then $(\mathcal X,L^\infty)$ is a complete metric space. Since $\mathcal U$ is bounded in 
$W^{1,\infty}(\Omega_T;\R^m)$, it follows that $\mathcal X\subset W_{\bar u}^{1,\infty}(\Omega_T;\R^m)$.   By the density assumption, $\mathcal U_\epsilon$ is also dense in $(\mathcal X,L^\infty).$ Let  $\mathcal Y=L^2(\Omega_T;(\R^{n+1})^m)$; then $(\mathcal Y,L^2)$ is a complete metric space as well. Since $\mathcal X\subset W_{\bar u}^{1,\infty}(\Omega_T;\R^m)$, the operator 
\[
\nabla =(D,\partial_t)\colon \mathcal X\to \mathcal Y
\]
is well-defined. Given $u\in \mathcal X$, write $u=\bar u+\phi$, where $\phi\in W_0^{1,\infty}(\Omega_T;\R^m)$ extended to be zero outside of $\Omega_T.$ Let $\rho_k\in C^\infty_c(\R^{n+1})$   be the standard mollifiers on $\R^{n+1}$ as $k\to \infty$ and let $\nabla_k u=\nabla \bar u + \nabla (\rho_k * \phi).$ Then
\[
\nabla_k u= \nabla \bar u + \rho_k * \nabla\phi= \nabla \bar u + \phi * \nabla \rho_k.
\]
Then $\|\nabla_k u_1-\nabla_k u_2\|_{L^2} \le C_k \|u_1-u_2\|_{L^\infty}$ for all $u_1,u_2\in \mathcal X$; hence $\nabla_k\colon \mathcal X\to \mathcal Y$ is continuous; furthermore, $\|\nabla_k u- \nabla u\|_{L^2}\to 0$ as $k\to \infty$  for all $u\in \mathcal X.$ This proves that $\nabla\colon (\mathcal X,L^\infty)\to (\mathcal Y,L^2)$ is the pointwise limit of a sequence of continuous functions $\nabla_k \colon (\mathcal X,L^\infty)\to (\mathcal Y,L^2)$ and thus 
 is a Baire-one function \cite{D}. Let $\mathcal G\subset \mathcal X$ be the set of continuity points of $\nabla.$ Then $\mathcal G$ is dense in $(\mathcal X,L^\infty).$ To complete the proof, we show that $\mathcal G\subset \mathcal S.$  
  Let $u\in \mathcal G.$ Since $\mathcal U_\epsilon$ is dense in $\mathcal X$, there exists $u^j\in \mathcal U_{1/j}\subset \mathcal X$ such that $\|u^j-u\|_{L^\infty}<1/j\to 0.$   Since $\nabla$ is continuous at $u\in\mathcal G$, we have $\|\nabla u^j- \nabla u\|_{L^2}\to 0.$  Hence $\sigma(Du^j)\to \sigma(Du) $ and $u^j_t\to u_t$ in $L^2(\Omega_T)$; thus
 \[
 \|u_t-\dv\sigma(Du)\|_{H^{-1}(\Omega_T)}=\lim_{j\to\infty} \|u^j_t-\dv\sigma(Du^j)\|_{H^{-1}(\Omega_T)}=0,
 \]
which proves that $u$ is a Lipschitz solution of (\ref{sys0}) and hence $u\in \mathcal S.$
 \end{proof}
 

   The $H^{-1}$-norm involved in Theorem \ref{baire-1} can be estimated as follows.
   
\begin{lem}\label{lem3}  Let $u\in W^{1,\infty}(\Omega_T;\R^m)$ and $v^i\in W^{1,\infty}(\Omega_T;\R^n)$ for $i=1,\dots,m.$  If $u^i=\dv v^i$ a.e. in $\Omega_T$ for each $i$, then
\[
\|u_t-\dv \sigma(Du)\|_{H^{-1}(\Omega_T)} \le C \|v_t-\sigma(Du) \|_{L^2(\Omega_T)}.
\]
\end{lem}
 
 \begin{proof} For all $\phi\in H^1_0(\Omega_T)$ and $i=1,\dots,m$, from $u^i=\dv v^i$, one has
\[
\langle u^i_t-\dv \sigma^i(Du),\phi\rangle =-\int_{\Omega_T} u^i\phi_t \,dxdt +\int_{\Omega_T} \sigma^i(Du)\cdot D\phi\,dxdt
\]
\[
=  \int_{\Omega_T} (\sigma^i(Du)-v^i_t)\cdot D\phi\,dxdt \le  \|v^i_t-\sigma^i(Du) \|_{L^2(\Omega_T)}\|D\phi\|_{L^2(\Omega_T)}.
\]
The result follows  from the definition of $
\|u^i_t-\dv \sigma^i(Du)\|_{H^{-1}(\Omega_T)}$.
\end{proof}

 Let $\sigma\colon \M^{m\times n}\to \M^{m\times n}$  satisfy   Condition  (OC) with open set $\Sigma \subset \M^{m\times n}\times (\R^n)^m$  as given in Definition \ref{cond-C}.  Let $\bar  u \in C^1(\bar\Omega_T;\R^m)$ and $\bar  v ^i\in C^1(\bar\Omega_T;\R^n)$ satisfy
\begin{equation}\label{suff-01}
\bar u^i=\dv \bar  v ^i,\quad [D\bar  u ,\,(\bar  v ^i_t)]\in \Sigma  \quad \mbox{on $\bar\Omega_T$}\quad (i=1,\dots,m).
\end{equation}
In order to construct  admissible subsolution sets $\mathcal U$ and $\mathcal U_\epsilon$ as required  in Theorem \ref{baire-1}, we employ the  piecewise $C^1$ functions on finitely many pieces.

\begin{defn} Let $E\subset \R^p$ be an open set and $f\colon E\to\R^q.$ 
We say  that $f$ is (finitely) {\em piecewise $C^1$ on  $E$} and write $f\in C^1_{piece}(E)$ if there exists a finite family of  disjoint open sets $\{E_1,\dots, E_\mu\}$ in $E$ such that
\[
|\partial E_j|=0,\quad f\in C^1(\bar E_j) \; \; \forall\, j=1,\dots,\mu, \quad |E\setminus \cup_{j=1}^\mu E_j|=0.
\]
In this case, we say that $\{E_1,\dots, E_\mu\}$ is a {\em partition} for $f$.
In the following, for simplicity, we set
\[
\begin{split} &C^1_{\bar  u }(\bar \Omega_T;\R^m)= W_{\bar  u }^{1,\infty}(\Omega_T;\R^m) \cap C^1(\bar\Omega_T),\\
&C^1_{\bar  v ^i, pc}(\Omega_T;\R^n)=W_{\bar  v ^i}^{1,\infty}(\Omega_T;\R^n) \cap C_{piece}^1(\Omega_T).\end{split}
\]
\end{defn}

\begin{thm}\label{density1} Let  $a>\|\bar  u _t\|_{L^\infty(\Omega_T)}$ be  fixed. Define
\[
\mathcal U=\left \{u \in C^1_{\bar  u }(\bar\Omega_T;\R^m) \; \Big | \; \begin{aligned} 
&\mbox{$\| u _t\|_{L^\infty(\Omega_T)}<a;\; \; \exists v ^i\in C^1_{\bar  v ^i, pc}(\Omega_T;\R^n)$ with}  \\ 
&\mbox{partition $\{E_j\}_{j=1}^\mu$ such that $u^i=\dv  v ^i$ and}\\
&\mbox{$[D u ,( v ^i_t)] \in\Sigma$ on  $\bar E_j \; \; \forall\, 1\le i\le m, \,1\le j \le \mu$}
\end{aligned}
\right\}
\]
and, for  $\epsilon>0,$ 
\[
\mathcal U_\epsilon=\left \{  u \in \mathcal U\;\Big | \; \begin{aligned} 
&\mbox{$\exists\,  v ^i\in C^1_{\bar  v ^i, pc}(\Omega_T;\R^n)$ with partition $\{E_j\}_{j=1}^\mu$ such that}\\
&\mbox{$u^i=\dv  v ^i,\; [D u ,( v ^i_t)] \in\Sigma $ on $\bar E_j, \; \| v ^i_t-\sigma^i(D u ) \|_{L^2(\Omega_T)}<\epsilon$} 
\end{aligned}
\right\}.
\]
Then
for each $\epsilon>0$ the set $\mathcal U_\epsilon$ is dense in $\mathcal U$ under the $L^\infty$-norm.  
\end{thm}

The proof of this density theorem will be given in the next final section. It is easily seen that our main theorem, Theorem \ref{thm-main-1},  follows by combining 
 Theorem \ref{baire-1} and Theorem \ref{density1}.

\section{The density result: Proof of Theorem \ref{density1}}
This final section is devoted to the proof of the density result: Theorem \ref{density1}. We first discuss some useful constructions. 
Let $\tilde Q_0=(0,1)^n\subset \R^n$ and $Q_0=\tilde Q_0\times (0,1)\subset \R^{n+1}$ be the open unit cubes.  For $y\in \R^{n+1}$ and $l>0$, let $Q_{y,l}=y+lQ_0.$ 
 
The following result  can be found in {Kim \&  Yan} \cite{KY1}.

\begin{lem}[{\cite[Theorem 2.3]{KY1}}]\label{div-inv}   Let $u\in W^{1,\infty}_0(Q_0)$ with  $\int_{\tilde Q_0} u (x,t) dx=0$ for all $t\in (0,1).$ Then there exists a function $ v =\mathcal R  u\in W^{1,\infty}_0(Q_0;\R^n)$ such that $\dv  v =u$ a.e.\,in $Q_0$ and
\begin{equation}\label{div-1}
\| v _t \|_{L^\infty(Q_0)}  \le C_n  \|u_t\|_{L^\infty(Q_0)},
\end{equation}
where $C_n$ is a constant depending only on $n$.   Moreover, if in addition $u\in C^1(\bar Q_0)$ then $ v =\mathcal R  u \in C^1(\bar Q_0;\R^{n}).$
\end{lem}
  
Given $\bar y\in \R^{n+1}$ and $l>0,$ let  $\mathcal L_{\bar y,l}  \colon W^{1,\infty}(Q_0;\R^d)\to W^{1,\infty}(Q_{\bar y, l};\R^d)$ be the {\em rescaling operator}  defined  by
 \[
(\mathcal L_{\bar y,l} f)(y)=lf(\frac{y-\bar y}{l})\quad (y\in Q_{\bar y,l})
\]
for all $f\in W^{1,\infty}(Q_0;\R^d).$ Note that $\nabla (\mathcal L_{\bar y,l} f)(y)=(\nabla f)(\frac{y-\bar y}{l})$ for $y\in Q_{\bar y, l}.$

From Lemma \ref{div-inv}, the following result is immediate.
 
 \begin{cor}\label{div-inv-1}  Let $\phi\in W^{1,\infty}_0(Q_{0})$ with $\int_{\tilde Q_0} \phi (x,t)\,dx=0$ for all $t\in (0,1)$. Let $\tilde \phi =\mathcal L_{\bar y,l}\phi$ and 
 $\tilde g=  \mathcal R_{\bar y, l} \phi:=l \mathcal L_{\bar y,l} (\mathcal R \phi)$ in $W^{1,\infty}_0(Q_{\bar y, l};\R^n).$
  Then $\dv \tilde g=\tilde \phi$ a.e.\,in $Q_{\bar y,l}$ and
\begin{equation}\label{div-2}
\|\tilde g_t \|_{L^\infty(Q_{\bar y, l})}  \le C_n l  \|\tilde \phi_t\|_{L^\infty(Q_{\bar y, l})}.
\end{equation}
Moreover, if in addition $\phi\in C^1(\bar Q_{0})$ then $\tilde g =\mathcal R_{\bar y, l} \phi \in C^1(\bar Q_{\bar y, l};\R^{n}).$
\end{cor}

\subsection*{Proof of Theorem \ref{density1}}  Let $\mathcal U$ and $\mathcal U_\epsilon$ be defined as above. 
Clearly $\bar  u \in\mathcal U;$ so $\mathcal U\ne \emptyset.$ Since $\Sigma$ is bounded, let
 \begin{equation} \label{boundM}
\begin{split} &M:=\sup_{[A,(b^i)]\in\Sigma} (|A|+|\sigma(A)|+|b^i|)  <\infty,\\
&\tilde M=\max_{|A|\le 1+3M} |\sigma(A)|,\quad
 \alpha(s)=\sup_{|A|,|A'|\le 1+3M\atop |A-A'|<s} |\sigma(A)-\sigma(A')|.
\end{split}
\end{equation}
The uniform continuity of $\sigma$ on $|A|\le 1+3M$ implies that  
$\alpha(s)\to 0$ as $s\to 0^+.$ Moreover, $\mathcal U$ is a bounded set  in $W_{\bar  u }^{1,\infty}(\Omega_T;\R^m)$ with 
\begin{equation}\label{bd-U}
\| u _t\|_{L^\infty(\Omega_T)}+\|D u \|_{L^\infty(\Omega_T)}\le a+M \quad\forall\;  u \in\mathcal U.
\end{equation}

Let $\epsilon>0, \,  u \in \mathcal U$ and $\rho>0$ be fixed. Then 
$\| u _t\|_{L^\infty(\Omega_T)}<a$ and there exist functions  $ v ^i \in C^1_{\bar  v ^i,pc}(\Omega_T;\R^n)$ with   partition $\{E_j\}_{j=1}^\mu$ such that 
\[
u^i=\dv  v ^i,\quad [D u,( v ^i_t)] \in\Sigma  \;\;\; \mbox{on $\bar E_j$ \;\; ($i=1,\dots,m;\; j=1,\dots,\mu$).}
\]
The goal is to construct a function  $\tilde  u \in\mathcal U_\epsilon$ such that $\|\tilde  u - u \|_{L^\infty(\Omega_T)}<\rho;$ that is,    we are to construct a function $\tilde u \in  C_{\bar u } ^1(\bar\Omega_T;\R^m)$ such that
\begin{itemize}
\item[(i)]    $\|\tilde  u _t\|_{L^\infty(\Omega_T)}<a$ and $\|\tilde  u - u \|_{L^\infty(\Omega_T)}<\rho;$  
\item[(ii)]  $\exists\, \tilde  v ^i\in C^1_{\bar  v ^i,pc}(\Omega_T;\R^n)$ and some partition $\{P_j\}_{j=1}^\kappa$ with
 \begin{equation}\label{pf-need}
 \begin{cases}   \tilde u^i=\dv \tilde  v ^i \quad \mbox{on each $\bar P_j$,}\\
  [D\tilde  u , (\tilde  v ^i_t)] \in\Sigma \quad \mbox{on each $\bar P_j$,}\\
   \|\tilde  v ^i_t-\sigma^i(D\tilde  u ) \|_{L^2(\Omega_T)}<\epsilon.
\end{cases}
\end{equation}
\end{itemize}

We  accomplish  the construction of $\tilde u $ in several steps. For simplicity, in what follows, we write $y=(x,t).$

\subsection*{Step 1}  
Fix $\nu  \in \{1,\dots,\mu\}$ and  $\bar y\in  E_\nu.$ Let
$A= D u (\bar y)$ and $b^i=   v ^i_t(\bar y);$ then  $[A,(b^i)] \in \Sigma.$ By Condition (OC),  there exists a $\tau_N$-configuration $(\xi_1,\xi_2,\dots,\xi_N)$  in $\mathbb K$  given by $\rho=[\tilde A,(\tilde b^i)], \gamma_j=[p_j\otimes \alpha_j,(s_j\beta_j^i)]$ and $\kappa_j>1$  such that
$
[A,(b^i)]\in \tau(\xi_1,\dots, \xi_N)\subset \Sigma.$ With these $\tilde A, \tilde b^i, p_j,\alpha_j, s_j,  \beta_j^i$ and  $\kappa_j$, let $\tilde  P, \tilde P_j^s$ and $(\tilde X_1^s,\dots,\tilde X_N^s)$ be  defined as in Lemma \ref{lem-X-s}.   Let $\tau>0$ be sufficiently small so that, for
\begin{equation}\label{pm-4}
X_j^{s,\tau}=(1-\tau) \tilde X^s_j + \tau \tilde P^s_j \quad (j=1,\dots,N),
\end{equation}
the $N$-tuple $(X_1^{s,\tau},\dots,X_N^{s,\tau})$ is an admissible $T_N$-configuration and  that $[A,(b^i)]\in \mathbb P( T (X_1^{s,\tau},\dots,X_N^{s,\tau}))$; note that $\mathbb P(X_j^{s,\tau})=\mathbb P(X_j^{1,\tau})$ for all $s\ne 0.$ Since
\[
\lim_{\tau\to 0^+} \dist(\mathbb P(X_j^{1,\tau}); \mathbb K)= \dist(\mathbb P(\tilde X^1_j); \mathbb K)= \dist(\xi_j; \mathbb K)=0,
\]
there exists  a further smaller $\tau>0$ such that
\begin{equation}\label{tau1}
\dist(\mathbb P( X_j^{1,\tau}); \mathbb K)<\frac{\epsilon}{8(|\Omega_T|)^{1/2}}\;\; (j=1,\dots,N).
\end{equation}
Fix such a $\tau>0.$ Then 
\[
\mathbb P(\bar { T }(X_1^{1,\tau},\dots,X_N^{1,\tau})) \subset \mathbb P({ T }(X_1,\dots,X_N)) \subset \Sigma.
\]
Since $\Sigma$ is open and $\mathbb P(\bar { T }(X_1^{1,\tau},\dots,X_N^{1,\tau}))$ is compact,  there exists a number  $\delta_\tau>0$  such that
\[
 [\mathbb P(\bar { T }(X_1^{1,\tau},\dots,X_N^{1,\tau}))]_{\delta_\tau}\subset \Sigma.
\]
 Hence, for all $s\ne 0$, 
  \begin{equation}\label{tau2}
\mathbb  P([\bar{ T } (X_1^{s,\tau},\dots,X_N^{s,\tau})]_{\delta_\tau})\subset  [\mathbb P(\bar { T }(X_1^{s,\tau},\dots,X_N^{s,\tau}))]_{\delta_\tau}\subset \Sigma.
 \end{equation}

\subsection*{Step 2} Apply Theorem \ref{thm1} to the unit cube  $G=Q_0\subset \R^{n+1}$ with $\tilde P\in  T ( X_1^{s,\tau},\dots, X_N^{s,\tau})$ to obtain  a function $\omega=[\varphi, (\psi^i)]\in C_c^{\infty}(Q_{0};\R^m \times (\R^n)^m)$ such that
\begin{itemize}
\item[(a)] \; $\dv \psi^i=0$  in $\bar Q_{0}$,

\item[(b)] \; $|\{y\in Q_{0} : [A+D\varphi(y),(b^i +\psi^i_t(y))]\notin \{\xi_1,\dots,\xi_N\}|<\epsilon',$

\item[(c)] \; $ [A+D\varphi(y),(b^i +\psi^i_t(y))]\in \mathbb P([\bar{ T } (X_1^{s,\tau},\dots,X_N^{s,\tau})]_{\epsilon'})$ for all $y\in Q_{0},$

\item[(d)] \; $\|\varphi_t\|_{L^\infty(Q_{0})} < \epsilon' +M'|s|; \; \|\varphi\|_{L^\infty(Q_{0})}<\epsilon',$

\item[(e)] \; $\int_{\tilde Q_{0}}\varphi(x,t)\,dx=0$ for all $t\in (0,1),$
\end{itemize}
where  $M'>0$ is a constant depending on $[A,(b^i)],$ and $\epsilon'\in (0,1)$ is  a number to be chosen later.

\subsection*{Step 3} Let $0<l<1$. Consider the functions $[\tilde \varphi,(\tilde \psi^i)]= \mathcal L_{\bar y, l} [\varphi,(\psi^i))]$ and $\tilde  g^i=\mathcal R_{\bar y,l}  \varphi^i$ defined on $Q_{\bar y,l},$   where $\mathcal L_{\bar y,l}$ and $\mathcal R_{\bar y,l}$ are the operators defined above.  
 Let
\begin{equation}\label{new-fun}
\tilde u =  u _{\bar y,l}= u +\tilde \varphi,\;\; \tilde  v ^i =  v ^i_{\bar y,l}= v ^i + \tilde \psi^i + \tilde {g}^i \quad \mbox{ on $Q_{\bar y,l}.$}
 \end{equation}
 Then, $\tilde u \in  u +C_c^\infty(Q_{\bar y,l})$, $\tilde  v ^i \in W^{1,\infty}_{ v ^i}(Q_{\bar y,l})\cap C^1(\bar Q_{\bar y,l}),$ and  $\dv \tilde  v ^i =\tilde u^i$ on $\bar Q_{\bar y,l}.$ 
By (c) and (d) of Step 2,  (\ref{boundM}),  and Corollary \ref{div-inv-1}, we have
\begin{equation}\label{d-1}
\begin{cases}\|\tilde   u - u \|_{L^\infty(Q_{\bar y,l})}=\|\tilde \varphi\|_{L^\infty(Q_{\bar y,l})}<l\epsilon'<\epsilon',\\
 \|\tilde   u _t\|_{L^\infty(Q_{\bar y,l})}<\| u _t\|_{L^\infty(\Omega_T)} +\epsilon'+M'|s|,\\
 \|\tilde  g^i_t\|_{L^\infty(Q_{\bar y,l})}\le  C_n l (\epsilon'+M'|s|),\\
 \|D\tilde \varphi \|_{L^\infty(Q_{\bar y,l})}\le \epsilon'+M,\\
\|\tilde \psi^i_t\|_{L^\infty(Q_{\bar y,l})}\le \epsilon'+ M.
\end{cases}
\end{equation}

\subsection*{Step 4} In this step, we estimate $\|\tilde  v ^i_t-\sigma^i(D\tilde  u )\|_{L^2(Q_{\bar y,l})}.$ Note that
\[
\begin{split} 
\|\tilde  v ^i_t  - &\sigma^i(D\tilde  u )\|_{L^2(Q_{\bar y,l})}  =\|  v ^i_t+\tilde \psi^i_t+ \tilde  g^i_t-\sigma^i(D u +D\tilde \varphi)\|_{L^2(Q_{\bar y,l})}\\
&\le   \| v ^i_t-b^i \|_{L^2(Q_{\bar y,l})}+\|b^i +\tilde \psi^i_t-\sigma^i(A+D\tilde \varphi)\|_{L^2(Q_{\bar y,l})}\\
&+ \|\tilde  g^i_t\|_{L^2(Q_{\bar y,l})}+ \|\sigma^i(A+D\tilde \varphi)-\sigma^i(D u +D\tilde \varphi)\|_{L^2(Q_{\bar y,l})}.
\end{split}
\]
By (\ref{d-1}),
\[
\|\tilde  g^i_t\|_{L^2(Q_{\bar y,l})}\le  C_n l (\epsilon'+M'|s|)|Q_{\bar y,l}|^{1/2}.
\]
Note that
\[
\|b^i+\tilde \psi^i_t-\sigma^i(A+D\tilde \varphi)\|^2_{L^2(Q_{\bar y,l})}=\int_{F\cup G} |b^i+\tilde \psi^i_t-\sigma^i(A+D\tilde \varphi)|^2\,dy,
\]
where $G=Q_{\bar y,l}\setminus F$ and
\[
F=\{y\in Q_{\bar y,l}\, |\,[A+D\tilde \varphi(y),(b^i+\tilde \psi^i_t(y))]\notin \{\cup_{j=1}^N \mathbb P(X_j)\}\}.
\]
 By (b) of Step 2, $|F|<\epsilon'|Q_{\bar y,l}|$ and, by (\ref{d-1}),  $|A+D\tilde \varphi|\le 1+3M$ and $|D u +D\tilde \varphi|\le 1+3M$ on $Q_{\bar y,l}.$  Hence
 \[
 \int_{F} |b^i+\psi^i_t-\sigma^i(A+D\varphi)|^2\,dy<\epsilon' (1+3M+\tilde M)^2|Q_{\bar y,l}|.
 \]
 By (\ref{tau1}),
 \[
  \int_{G} |b^i+\psi^i_t-\sigma^i(A+D\varphi)|^2\,dy\le \frac{\epsilon^2}{32|\Omega_T|} |Q_{\bar y,l}|.
  \]
 Hence
\[
\|b^i+\tilde \psi^i_t-\sigma^i(A+D\tilde \varphi)\|^2_{L^2(Q_{\bar y,l})}\le \left [(1+3M+\tilde M)\sqrt{\epsilon'} + \frac{\epsilon}{4(|\Omega_T|)^{1/2}}\right ] |Q_{\bar y,l}|^{1/2}.
 \]
 Let
\[
m(l)=\max_{1\le i\le m;\; y\in Q_{\bar y,l}} \left ( | v ^i_t(y)-b^i|+|D u (y)-A|\right ).
\]
Then $m(l)\to 0$ as $l\to 0^+.$ We have the following estimates:
\[
\| v ^i_t-b^i\|_{L^2(Q_{\bar y,l})}\le m(l) |Q_{\bar y,l}|^{1/2};
\]
\[
\|\sigma^i(A+D\tilde \varphi)-\sigma^i(D u +D\tilde \varphi)\|_{L^2(Q_{\bar y,l})}\le \alpha(m(l))|Q_{\bar y,l}|^{1/2},
\] 
where $\alpha(s)$ is the function defined in (\ref{boundM}). Hence, we obtain 
  \begin{equation}\label{L2-est}
 \begin{split}
   \|&\tilde  v ^i_t -\sigma^i(D\tilde  u )\|_{L^2(Q_{\bar y,l})}   \le  \Big [(1+3M+\tilde M)\sqrt{\epsilon'} + C_nl \epsilon' \\
 +&m(l)+\alpha(m(l))+ 2MC_nl  |s| +\frac{\epsilon}{4(|\Omega_T|)^{1/2}} \Big ]|Q_{\bar y,l}|^{1/2}. 
  \end{split}
 \end{equation}
 
 \subsection*{Step 5} In this step, we estimate 
 \[
 \dist([D\tilde  u ,(\tilde  v ^i_t)];\,\mathbb P(\bar{ T }(X_1^{1,\tau},\dots,X_N^{1,\tau})))
 \]
  on $Q_{\bar y,l}$. Since $D\tilde u =D u +D\varphi$ and $\tilde  v ^i_t= v ^i_t+\tilde \psi^i_t+\tilde  g^i_t$, we have on $Q_{\bar y,l}$,
 \[
 \begin{split} &\dist([D\tilde  u , (\tilde  v ^i_t)];\,\mathbb P(\bar{ T }(X_1^{1,\tau},\dots,X_N^{1,\tau})))\\
   \le  \dist([A&+D\tilde \varphi,(b^i + \tilde\psi^i_t)];\, \mathbb P(\bar{ T }(X_1^{1,\tau},\dots,X_N^{1,\tau}))) +|[D u -A, ( v ^i_t-b^i +\tilde  g^i_t)]|\\
  \le  \dist([A&+D\tilde \varphi,(b^i + \tilde\psi^i_t)];\, \mathbb P(\bar{ T }(X_1^{1,\tau},\dots,X_N^{1,\tau}))) +|D u -A| +|( v ^i_t-b^i)|+|\tilde  g^i_t|.
\end{split}
 \]
 By (c) of Step 2 and (\ref{d-1}), we have 
 \begin{equation}\label{dist-est}
 \begin{split}
 \dist([D\tilde  u ,(\tilde  v ^i_t)];\, &\mathbb P(\bar{ T }(X_1^{1,\tau},\dots,X_N^{1,\tau})))\\
 &< (1+C_n l)\epsilon' + 2m(l) + 2MC_n l|s|. \end{split}
  \end{equation}
    
    \subsection*{Step 6} In this step, we select the small numbers $\epsilon'\in (0,1)$ and $s\ne 0$ in the estimates (\ref{d-1}), (\ref{L2-est}) and (\ref{dist-est}) to ensure that, for all sufficiently small $l\in (0,1)$, it holds that
\begin{equation}\label{pf8}
\begin{cases}
\|\tilde  u - u \|_{L^\infty(Q_{\bar y,l})} <\rho,\\
\|\tilde  u _t\|_{L^\infty(Q_{\bar y,l})} <a,\\
    [D\tilde  u ,(\tilde  v ^i_t)] \in \Sigma \;\;\mbox{on $Q_{\bar y,l},$}\\
   \|\tilde  v ^i_t-\sigma^i(D\tilde  u )\|_{L^2(Q_{\bar y,l})}<\dfrac{\epsilon}{2(|\Omega_T|)^{1/2}} |Q_{\bar y,l}|^{1/2}.
   \end{cases}
   \end{equation}
   To do so, in  view of   (\ref{d-1}), (\ref{L2-est}) and  (\ref{dist-est}), we select $\epsilon'$ and $s$ such that 
  \[
  \begin{cases}
  0<\epsilon'<\min\left \{1,\, \rho, \frac{a-\|u_t\|_{L^\infty(\Omega_T)}}{2},\frac{\delta_\tau}{3(1+C_n)}\right \},\\
    (1+3M+\tilde M)\sqrt{\epsilon'}+C_n\epsilon' <\frac{\epsilon}{16(|\Omega_T|)^{1/2}},\\
   0<|s|<\min\left\{\frac{a-\|u_t\|_{L^\infty(\Omega_T)}}{4M},\; \frac{\epsilon}{32(|\Omega_T|)^{1/2}C_n M},\; \frac{\delta_\tau}{6C_n M}\right\}.
   \end{cases}
  \]
  This guarantees the first two requirements in (\ref{pf8}). 
  Also, by (\ref{L2-est}) and (\ref{dist-est}),  
  \[
     \|\tilde  v ^i_t -\sigma^i(D\tilde  u )\|_{L^2(Q_{\bar y,l})}   < \left [ \frac{3\epsilon}{8(|\Omega_T|)^{1/2}}  +m(l)+\alpha(m(l))\right ] |Q_{\bar y,l}|^{1/2},
     \]
     \[
 \dist([D\tilde  u ,(\tilde  v ^i_t)];\,\mathbb P(\bar{ T }(X_1^{1,\tau},\dots,X_N^{1,\tau})))
  <\frac{2\delta_\tau}{3} +2m(l).
  \]
Since $m(l)\to 0$ and $\alpha(m(l))\to 0$ as $l\to 0^+$ and $E_j$ is open, there exists a number $0<l_{\bar y}<1$ such that, for all $0<l<l_{\bar y}$,
     \[
   Q_{\bar y,l}\subset E_j,\quad   m(l)+\alpha(m(l))<\frac{\epsilon}{8(|\Omega_T|)^{1/2}},\quad m(l)<\frac{\delta_\tau}{6},
     \]
 where $\delta_\tau>0$ is the number  defined by (\ref{tau2}).  With  all such $l$'s, the last two requirements in (\ref{pf8}) are also satisfied on $Q_{\bar y,l}$: for example, from
\[
\dist([D\tilde  u ,(\tilde  v ^i_t)];\,\mathbb P(\bar{ T }(X_1^{1,\tau},\dots,X_N^{1,\tau})))
  <\delta_\tau \quad \mbox{on $Q_{\bar y,l}$}
  \]
  and (\ref{tau2}), we have
      \[
  [\tilde  u ,(\tilde  v ^i_t)] \in [\mathbb P(\bar{ T }(X_1^{1,\tau},\dots,X_N^{1,\tau}))]_{\delta_\tau}\subset \Sigma \quad \mbox{on $Q_{\bar y,l}.$}
  \]

\subsection*{Step 7} For the fixed $\nu$, the family $\{\bar Q_{\bar y,l}\, |\;\bar y\in E_\nu,\; 0<l<l_{\bar y}\}$ forms a Vitali covering of the set $E_\nu$ by closed cubes (see \cite{DM}). Thus,  there exists a countable subfamily of disjoint  closed cubes $\{P_{\nu,k}=\bar Q_{\bar y_k,l_k} \;|\, k=1,2,\dots\}$ such that
  \[
  E_\nu=(\cup_{k=1}^\infty P_{\nu,k})\cup R_\nu,\quad |R_\nu|=0.
  \]
 Let $\tilde  u _{\nu,k}=  u _{\bar y_k,l_k}$ and $\tilde  v ^i_{\nu,k}= v ^i_{\bar y_k,l_k}$ be the functions defined by (\ref{new-fun}) on $P_{\nu,k}=\bar Q_{\bar y_k,l_k}.$
  For each $\nu=1,\dots,\mu$, let $N_\nu$ be such that
 \begin{equation}\label{L2-0}
  \left | \cup_{k=N_\nu+1} ^\infty P_{\nu,k}\right |=\sum_{k=N_\nu+1}^\infty |P_{\nu,k}| <\frac{\epsilon^2}{2\mu M^2}.
 \end{equation}
Consider  the partition $\Omega_T= \left (\cup_{\nu=1}^\mu \cup_{k=1}^{N_\nu} P_{\nu,k} \right)\cup P,$ 
 where 
 \[
 P=\Omega_T\setminus \left (\cup_{\nu=1}^\mu \cup_{k=1}^{N_\nu} P_{\nu,k} \right) =\left (\cup_{\nu=1}^\mu \cup_{k=N_\nu+1}^\infty P_{\nu,k} \right) \cup R
 \]
with $|R|=0.$  With this partition, we define 
 \[
\tilde  u =u \chi_{P}  + \sum_{\nu=1}^\mu \sum_{k=1}^{N_\nu} \tilde  u _{\nu,k} \chi_{P_{\nu,k}},\quad 
\tilde  v ^i= v^i \chi_{P}  + \sum_{\nu=1}^\mu \sum_{k=1}^{N_\nu} \tilde  v ^i_{\nu,k}  \chi_{P_{\nu,k}}.
\]
Then $\tilde  u - u \in C^\infty_c (P_{\nu,k})$ and $\tilde  v ^i- v ^i\in C^1(P_{\nu,k})$ for all $\nu.$  It is straightforward to see that $\tilde u \in W_{\bar  u }^{1,\infty}(\Omega_T)\cap C^1(\bar\Omega_T;\R^m)$ and $\tilde  v ^i\in C^1_{\bar  v ^i,pc}(\Omega_T;(\R^n)^m)$ with partition being
$\{P,\;P_{\nu,k} \; |\; \nu=1,\dots,\mu, \; k=1,\dots,N_\nu\}.$
Moreover, all requirements in (i) and (ii) at the start of the proof are  satisfied; for example,  the last one in (\ref{pf-need}) is proved  as follows.
\[
\begin{split}
&\qquad \qquad \qquad \|\tilde  v ^i_t-  \sigma^i(D\tilde  u ) \|_{L^2(\Omega_T)}^2  \\
&=\sum_{\nu=1}^\mu  \sum_{k=1}^{N_\nu} \|\tilde  v ^i_t- \sigma^i(D\tilde  u )\|_{L^2(P_{\nu,k})}^2   + \sum_{\nu=1}^\mu \sum_{k=N_\nu+1}^\infty \| v ^i_t- \sigma^i(D u )\|_{L^2(P_{\nu,k})}^2  \\
&\le \sum_{\nu=1}^\mu\sum_{k=1}^{N_\nu} \frac{\epsilon^2}{4|\Omega_T|} |P_{\nu,k}| + \sum_{\nu=1}^\mu\sum_{k=N_\nu+1}^\infty M^2 |P_{\nu,k}|\le \frac{\epsilon^2}{4|\Omega_T|} |\Omega_T| +   \frac{\mu M^2 \epsilon^2}{2\mu M^2}< \epsilon^2.
\end{split}
\]

Finally the proof of Theorem \ref{density1} is completed.

\end{document}